\documentclass[10pt]{article}

\usepackage[utf8]{inputenc}
\usepackage{hyperref}

\usepackage{amsmath,amsthm} 
\usepackage{amssymb,mathrsfs} 
\usepackage{a4wide} 
\usepackage{enumerate}
\usepackage[normalem]{ulem}
\usepackage{cancel}
\usepackage{bbm}
\usepackage{nicematrix}
\usepackage{nccmath}
\usepackage{mathtools}
\usepackage{bm}
\usepackage{xcolor}
\usepackage{relsize}
\usepackage{tikz-cd}

\newcommand{\cL}{\mathcal{L}}
\newcommand{\cR}{\mathcal{R}}
\renewcommand{\d}{\mathrm{d}}
\newcommand{\e}{\mathrm{e}}
\newcommand{\E}{\mathbb{E}}
\renewcommand{\P}{\mathbb{P}}
\renewcommand{\div}{\mathrm{div}}
\newcommand{\R}{\mathbb{R}}

\newcommand{\cX}{\mathcal{X}}
\newcommand{\cA}{\mathcal{A}}

\newcommand{\cC}{\mathcal{C}}
\newcommand{\cS}{\mathcal{S}}

\newcommand{\cP}{\mathcal{P}}
\newcommand{\cLham}{\mathcal A}
\newcommand{\cLou}{\mathcal S}

\newcommand{\cLod}{\cL_{\mathrm{KS}}}
\newcommand{\cSpd}{\mathcal{S}_d^{\scalebox{0.6}{++}}}
\newcommand{\wLmu}{L^2(\mu)}
\newcommand{\wLmuz}{L^2_0(\mu)}
\newcommand{\subplus}{\textnormal{\texttt{+}}}

\newcommand{\N}{\mathbb N}
\newcommand{\1}{\mathbbm 1}

\newcommand{\testfuncs}{\mathcal C^\infty_{\mathrm c}}
\newcommand{\dsint}{\displaystyle\int}

\newcommand{\Id}{\mathrm{Id}}

\newcommand{\cW}{\mathcal{W}}

\newcommand{\bigo}{\mathcal{O}}

\newtheorem{lemma}{Lemma}

\newtheorem{theorem}{Theorem}
\newtheorem{corollary}{Corollary}

\newtheorem{hypothesis}{Assumption}

\title{Overdamped limits for Langevin dynamics with position-dependent coefficients via $L^2$-hypocoercivity} 
\author{No\'e Blassel\\
   \small{MatMat group, Institute of Mathematics}\\\small{EPFL, Lausanne, Switzerland}  \\
   \small{\href{mailto:noe.blassel@epfl.ch}{noe.blassel@epfl.ch}}
}

\begin{document}
\maketitle

\begin{abstract}
    This note provides a simple derivation of the overdamped approximation for kinetic (or underdamped) equilibrium Langevin dynamics, in cases where certain coefficients depend on the position variable.
    The equivalent small-mass limit of these dynamics, known as the Kramers--Smoluchowski approximation, in the case of a state-dependent friction coefficient, has been previously studied by a variety of approaches.
    Our new approach uses hypocoercivity estimates, which may be of interest in their own right, and lead to a very direct derivation, providing in particular a clear explanation of the ``noise-induced drift'' term in the overdamped equation in the case of a state-dependent friction term.
    Using the same approach, we also treat effective kinetic dynamical models derived from a coarse-graining approximation of a high-dimensional system, as well as a class of kinetic dynamics with position-dependent mass matrices. All of these models are relevant to applications in computational chemistry.
    We finally identify a mistake in~\cite{WSW24} and suggest a solution.
\end{abstract}

\section{Introduction}
\label{sec:intro}

We study the overdamped limits of several kinetic Langevin equations, a typical example of which is given by the solution to the stochastic differential equation (SDE)
\begin{equation}
    \label{eq:underdamped_friction_dependent}
    \left\{
    \begin{aligned}
            \d q_t^\lambda &= \nabla U(p_t^\lambda) \d t,\\
            \d p_t^\lambda &= -\nabla V(q_t^\lambda) \d t - \lambda D(q_t^\lambda)^{-1}\nabla U(p_t^\lambda)\,\d t + \sqrt{\frac{2\lambda}{\beta}}D(q_t^\lambda)^{-1/2}\,\d W_t^\lambda,
        \end{aligned}
    \right.
\end{equation}
where $D:\cX\to \cSpd$ is a symmetric positive-definite matrix field called the inverse friction profile, and $W^\lambda$ is a standard Brownian motion in $\R^d$ given by a diffusive rescaling of some reference Brownian motion $W$, i.e. $W_t^\lambda = \sqrt{\lambda}W_{t/\lambda}$ for all~$\lambda,t>0$.

Note that instead of~$D^{-1/2}$ in~\eqref{eq:underdamped_friction_dependent}, we could take any matrix-field $A:\cX\to\R^{d\times d}$ such that~$AA^\top=D^{-1}$, with sufficient regularity (see Assumption~\ref{hyp:smoothness} below),
but since any two square roots of~$D^{-1}$ are conjugated via an orthogonal transformation, we may assume that~$A=D^{-1/2}$ is the positive square root of~$D^{-1}$, up to a change of Brownian motion.

In~\eqref{eq:underdamped_friction_dependent}, the functions~$U$ and~$V$ correspond respectively to the kinetic and potential energies, which sum to the Hamiltonian
\begin{equation}
    \label{eq:hamiltonian}
    H(q,p) = V(q)+U(p).
\end{equation}
The position variable~$q$ evolves in~$\mathcal X$, where $\cX$ is either the $d$-dimensional torus~$L\mathbb T^d$ for some~$L>0$, or~$\cX=\R^d$, while the momentum variable~$p\in\R^d$ is in all cases unbounded. The typical choice of kinetic energy is given by the physical definition
\begin{equation}
    \label{eq:kinetic_energy}
    U(p)=\frac12 p^\top M^{-1}p,
\end{equation}
where~$M$ is the mass matrix of the system, although other choices have been considered for applications to sampling high-dimensional measures, see for instance~\cite{ST18}.

We consider the large friction regime~$\lambda\to+\infty$, and to this effect, we introduce the following family of rescaled-in-time,~$W$-adapted processes:
\begin{equation}
    \label{eq:underdamped_time_rescaled}
    X_t^\lambda := q_{\lambda t}^\lambda,\qquad\forall\,\lambda>0,\,t\geq 0.
\end{equation}

Our main result, stated in Theorem~\ref{thm:overdamped_limit} below, states that trajectories of~$X^\lambda$ converge in the limit~$\lambda\to +\infty$ to solutions of the SDE
\begin{equation}
    \label{eq:overdamped_friction_dependent}
    \d X_t = -\left[D(X_t)\nabla V(X_t) - \frac1\beta\,\mathrm{div} D(X_t)\right]\,\d t + \sqrt{\frac{2}{\beta}}D^{1/2}(X_t)\d W_t,
\end{equation}
with initial data $X_0\in\cX$, and where~$\mathrm{div}$ denotes the row-wise divergence operator, which is well-defined under our regularity assumptions~(see Assumption~\ref{hyp:smoothness}). The dynamics~\eqref{eq:overdamped_friction_dependent} is the so-called Smoluchowski--Kramers dynamics. When~$D=\Id$ is the identity matrix, the dynamics~\eqref{eq:overdamped_friction_dependent} is also known as the overdamped Langevin equation.
In the case the kinetic energy is given by the physical choice~\eqref{eq:kinetic_energy}, the convergence actually holds pathwise (Corollary~\ref{corr:pathwise}), and the Brownian motion~$W$ in~\eqref{eq:overdamped_friction_dependent} is the same as the one appearing in~\eqref{eq:underdamped_friction_dependent}.
The term~$\frac1\beta\,\mathrm{div} D$, which is known as a ``spurious`` or ``noise-induced'' drift term in the experimental physics literature (see for example~\cite{VHBWB10,BVHWB11} and references therein).

In the case where the matrix~$D$ is a constant, the convergence of the dynamics~\eqref{eq:underdamped_friction_dependent} to the dynamics~\eqref{eq:overdamped_friction_dependent} is well known since the work of Smoluchowski~\cite{S16} and Kramers~\cite{K40}. Moreover, the large friction regime can be directly mapped to the small mass regime, via appropriate nondimensionalisations of the dynamics, see for instance~\cite[Section 2.2.4]{LRS10}.
The approximation of the dynamics~\eqref{eq:underdamped_friction_dependent} by the dynamics~\eqref{eq:overdamped_friction_dependent} in the small-mass regime is known as the Smoluchowski--Kramers approximation.
This approximation is much harder to obtain in the case of a position-dependent friction, which entails the noise in~\eqref{eq:underdamped_friction_dependent} and~\eqref{eq:overdamped_friction_dependent} is multiplicative, but has nevertheless been derived rigorously using a variety of methods (see~\cite{F04,FH12,HMVW15,WSW24} and references therein), including stochastic averaging and homogenization of PDEs approaches. The Kramers--Smoluchowski approximation has also been extended to more general settings, such as stochastic PDEs~\cite{C06}, kinetic diffusions on manifolds~\cite{BHVW17} and non-linear SDEs~\cite{LJW24}.

The purpose of this note is to provide a proof of the overdamped approximation (i.e. the limit $\lambda\to +\infty$) by different means than those previously used to study the small-mass limit, and for more general dynamics than~\eqref{eq:underdamped_friction_dependent}. We leverage analytical estimates from the theory of hypocoercivity~\cite{V06}, and more specifically the~$L^2$-hypocoercivity approach developed in~\cite{DKMS13,DMS15,BFLS22}.
Because of the structural assumption of~$L^2$-hypocoercivity, our result is for the moment limited to equilibrium systems, in which the force in~\eqref{eq:underdamped_friction_dependent} is a gradient and the fluctuation-dissipation relation holds between the kinetic damping term and the diffusion coefficient.
On the other hand, the result is quantitative in~$\lambda$, and valid for more general kinetic energies~$U$ than the physical kinetic energy~\eqref{eq:kinetic_energy} (at least at the level of time-marginals, see Theorem~\ref{thm:overdamped_limit} below for a precise statement). Besides, the proof is direct, and highlights the origin of the ``noise-induced drift'' in a straightforward way. We also identify a gap in the proof of~\cite[Lemma~3.1]{WSW24}, and supply a correct argument for a similar step in our setting.
The approach also generalizes to other underdamped dynamics, both relevant to molecular dynamics (MD) simulations: one is obtained by considering a low-dimensional effective kinetic model of some high-dimensional stochastic process (this case is studied in Theorem~\eqref{thm:effective_od}), and the other is a case where both the mass and the friction are matrix-valued position-dependent functions~(see Corollary~\ref{thm:overdamped_mass_matrix}).

The note is organized as follows. In Section~\ref{sec:notation}, we introduce the necessary notation and hypotheses, and state our main result, Theorem~\ref{thm:overdamped_limit}, which we prove in Section~\ref{sec:proof_main}, along with its Corollary~\ref{corr:pathwise}.
In Section~\ref{sec:coarse_graining}, we state and prove our second result (Theorem~\ref{thm:effective_od}), and discuss its implications for molecular modelling. We finally study in Section~\ref{sec:mass_matrices} cases with state-dependent mass matrices.
We finally gather in Appendix~\ref{sec:proofs} the proofs of some key technical results, including the hypocoercive estimates of Lemmas~\ref{lemma:hypocoercivity} and~\ref{lemma:hypocoercivity_bis}, which may be of independent interest.

\section{Notation and first main result}
\label{sec:notation}
In this section, we introduce the necessary notation, state our hypotheses and our main result, Theorem~\ref{thm:overdamped_limit} below.
\paragraph{Notation.}
Throughout this work, we abuse notation and denote probability measures with the same symbol as their densities whenever they are absolutely continuous with respect to the Lebesgue measure.
With this convention, the equilibrium (or Gibbs) measure~$\mu$ is defined (under Assumption~\ref{hyp:UVreg} below) by the probability measure
\begin{equation}
    \label{eq:boltzmann_gibbs}
    \mu(\d q\,\d p) = \mu(q,p)\,\d q\,\d p = \exp\left(-\beta H(q,p)\right)/Z_{\mu}\,\d q\,\d p,\qquad Z_{\mu}=\int_{\cX\times\R^d}\e^{-\beta H}.
\end{equation}
The measure~\eqref{eq:boltzmann_gibbs} is a product~$\mu(\d q\,\d p)=\kappa(\d p)\otimes\nu(\d q)$, with kinetic and configurational marginals
\begin{equation}
    \label{eq:gibbs}
    \kappa(p)\,\d p = \frac{\e^{-\beta U(p)}}{\dsint_{\R^d} \e^{-\beta U}}\,\d p,\qquad \nu(q)\,\d q = \frac{\e^{-\beta V(q)}}{\dsint_\cX \e^{-\beta V}}\,\d q.
\end{equation}
For any $f\in L^1(\cX\times \R^d,\mu)$, we define the projector
\begin{equation}
    \label{eq:kappa_proj}
    \Pi_0 f(q) = \int_{\R^d}f(q,p)\,\kappa(\d p)
\end{equation}
which sends~$f(q,p)$ to its partial average with respect to~$\kappa$. We will make frequent use of the weighted~$L^2$-space
\begin{equation}    
    \wLmu = \left\{f\in L^1_{\mathrm{loc}}(\cX\times\R^d)\,: \int_{\cX\times \R^d}f^2 \,\d \mu < +\infty\right\},
\end{equation}
and the closed subspace of~$\mu$-centered observables
\begin{equation}
    \wLmuz = \left\{f\in \wLmu:\,\int_{\cX\times\R^d}f\,\d\mu = 0\right\}.
\end{equation}
Note that~$\Pi_0$ is a~$\wLmu$-orthogonal projector, and that~${(\Id-\Pi_0)\wLmu \subset \wLmuz}$.
We define weighted Sobolev spaces, for~$k\geq 1$ by
\begin{equation}
    H^k(\mu) = \left\{\varphi\in L^2(\mu):\,\forall\,\alpha\in \N^d,\,|\alpha|\leq k,\,\partial^\alpha\varphi\in L^2(\mu)\right\},
\end{equation}
with the usual multi-index notation for~$\alpha\in \N^m$ i.e. $|\alpha|=\sum_{j=1}^m \alpha_j$, and~$\partial^\alpha f = \partial^{\alpha_1}_{x_1}\dots\partial_{x_m}^{\alpha_m}f$ for~$f:\R^m\to\R$.
Other weighted~$L^p$-spaces~$L^p(\mu)$,~$L^p(\nu)$,~$L^p(\kappa)$ and the associated Sobolev spaces are defined in a similar fashion.

The generator of the dynamics~\eqref{eq:underdamped_friction_dependent} acts on smooth functions as the differential operator
\begin{equation}
    \label{eq:generator}
    \cL_\lambda = \cLham + \lambda \cLou,
\end{equation}
where~$\cLham$ and~$\cLou$ are Hamiltonian transport and fluctuation-dissipation operators, defined respectively by
\begin{equation}
    \label{eq:generator_parts_tmp}
    \cLham = \nabla U(p)^\top\nabla_q - \nabla V(q)^\top\nabla_p,\qquad\cLou = -(D^{-1}(q)\nabla U(p))^\top\nabla_p + \frac1\beta D^{-1}(q):\nabla_p^2.
\end{equation}
The~$\wLmu$-adjoints of partial derivatives are found by integration by parts to be
\begin{equation}
    \label{eq:adjoints}
    \partial_{q_i}^* = -\partial_{q_i}+\beta\partial_{i}V(q),\qquad \partial_{p_i}^* = -\partial_{p_i}+\beta \partial_{i}U(p),\qquad\,\forall\, 1\leq i\leq d,
\end{equation}
from which we find
\begin{equation}
    \label{eq:generator_parts}
    \cLham = \frac1\beta\left[\nabla_p^{*}\nabla_q-\nabla_q^{*}\nabla_p\right],\qquad \cLou =-\frac1\beta \nabla_p^{*}D^{-1}(q)\nabla_p,
\end{equation}
so that $\lambda\cLou$ and~$\cLham$ correspond respectively to the symmetric and antisymmetric parts of~$\cL_\lambda$ on~$\wLmu$.

The generator of the dynamics~\eqref{eq:overdamped_friction_dependent}, on the other hand, is given by
\begin{equation}
    \cLod = -\nabla V^\top D\nabla + \frac1\beta\,\left(\mathrm{div}\, D\right)^\top \nabla + \frac1\beta D:\nabla^2 = \frac1\beta\e^{\beta V}\,\mathrm{div}\left(\e^{-\beta V}D\nabla\right),
\end{equation}
which is a symmetric operator on~$L^2(\nu)$, a property reflecting the reversibility of the dynamics~\eqref{eq:overdamped_friction_dependent} with respect to~$\nu$.

On the space~$\R^{d\times d}$ of~$d\times d$ matrices, we denote by~$\left\|\cdot\right\|_{\mathrm{HS}}$ the Hilbert--Schmidt (or Frobenius norm), and by~$\left\|\cdot\right\|_{\mathrm{op}}$ the Euclidean operator norm, as well as by~$L^p_{\mathrm{HS}}$ and~$L^p_{\mathrm{op}}$ the associated~$L^p$-spaces of matrix-valued functions for~$p\in [1,\infty]$.

The processes~$X,X^\lambda$ defined respectively in~\eqref{eq:overdamped_friction_dependent} and~\eqref{eq:underdamped_time_rescaled} are defined on some probability space~$(\Omega,\mathcal F,\P_{\mu_0})$,
and under the probability measure~$\P_{\mu_0}$,~$W$ is a Brownian motion independent from the initial conditions
\begin{equation}
    \label{eq:initial_conditions}
    (q_0^\lambda,p_0^\lambda)=(X_0,p_0)\,\qquad \forall\,\lambda >0,
\end{equation}
which in turn are distributed according to some probability measure~$\mu_0\in\cP(\cX\times\R^d)$. As we will consider various~$\mu_0$, this choice is emphasized in the notation~$\P_{\mu_0}$, and in the notation~$\E_{\mu_0}$ for the corresponding expectation functional.

\paragraph{Assumptions on~$U,V$ and~$D$.}
In this paragraph, we introduce various assumptions which we will use to state and prove our main result.

\begin{hypothesis}
    \label{hyp:smoothness}
    The coefficients of~\eqref{eq:underdamped_friction_dependent} are smooth
    \begin{equation}
        \label{eq:smoothness}
        V\in\mathcal C^\infty(\cX),\quad U\in\mathcal C^\infty(\R^d),\quad D\in\mathcal C^\infty(\cX;\cSpd).
    \end{equation}
    Furthermore, both~$U$ and~$V$ have precompact sublevel sets, and~$\nabla U,\nabla V$ are globally Lipschitz continuous.
    \end{hypothesis}

\begin{hypothesis}
    \label{hyp:hypoellipticity}
    For all~$\lambda>0$, the generator~$\cL_\lambda$ is hypoelliptic, in the sense that
    \begin{equation}
        \label{eq:hypoelliptic}
        g = \cL_\lambda f,\, g\in \mathcal C^\infty(\cX\times\R^d) \implies f\in \mathcal C^\infty(\cX\times \R^d).
    \end{equation}
\end{hypothesis}
Hypoellipticity can generally be checked using H\"ormander's bracket condition, see~\cite{H67} or Lemma~\ref{lemma:hypoellipticity} in Appendix~\ref{sec:proofs}, in which we give simple sufficient conditions on~$U$,~$V$ and~$D$ for the implication~\eqref{eq:hypoelliptic} to hold.

\begin{hypothesis}
    \label{hyp:elliptic}
    The diffusion~$D$ belongs to~$\cW^{2,\infty}(\cX;\cSpd)$ (i.e. it has bounded derivatives up to order~$2$) and is, moreover uniformly elliptic.
    
    In particular, the following holds:
    \begin{equation}
        \label{eq:elliptic}
        \exists\,M_D \geq 1:\forall\,q\in\cX, \frac1{M_D}\Id \leq D(q) \leq M_D \Id,\qquad D'\in L^\infty\left(\cX;\R^d\otimes \R^{d\times d}\right),
    \end{equation}
    where the inequalities are in the sense of symmetric matrices, and where~$D'$, defined by
    \begin{equation}
        \label{eq:diff_notation}
        \left(D'(q)\left[u\right] \right)_{ij} = \sum_{k=1}^d \partial_k D_{ij}(q)u_k,\qquad \forall (q,u)\in \cX\times\R^d,\,\forall\, 1\leq i,j\leq d,
    \end{equation}
    denotes the Fr\'echet differential of~$D$.
\end{hypothesis}
Note that the conjunction of Assumptions~\ref{hyp:smoothness} and~\ref{hyp:elliptic} implies that the same conditions on~$D^{-1}$ and~$D^{\pm1/2}$.
    
The next assumption (which contains in particular the conditions~\cite[Assumptions 3.1 and~3.2]{BFLS22}), ensures the validity of some key estimates (Lemma~\ref{lemma:hypocoercivity} below).
    \begin{hypothesis}
        \label{hyp:UVreg}
        The following conditions on~$U$ and~$V$ are satisfied.
        \begin{itemize}
        \item{It holds, for any~$\gamma>0$:
        \begin{equation}
            \label{eq:moment_bounds}
        (1+|p|^\gamma)\e^{-\beta U(p)}\in L^1(\R^d),\qquad(1+|q|^\gamma+|\nabla V(q)|^\gamma)\e^{-\beta V(q)}\in L^1(\cX).
        \end{equation}
         In particular, the equilibrium probability measure~\eqref{eq:gibbs} is well-defined.}
         \item{
          The marginals~$\kappa$ and~$\nu$ defined in~\eqref{eq:gibbs} satisfy Poincar\'e inequalities: there exist~$K_\kappa,K_\nu>0$ such that, for any~$\varphi\in H^1(\kappa)$ and~$\psi\in H^1(\nu)$,
        \begin{equation}
            \label{eq:poincare}
            \left\|\varphi-\kappa(\varphi)\right\|_{L^2(\kappa)}\leq \frac1{K_\kappa}\|\nabla_p\varphi\|_{L^2(\kappa)},\qquad \left\|\psi-\nu(\psi)\right\|_{L^2(\nu)}\leq \frac1{K_\nu}\|\nabla\psi\|_{L^2(\nu)},
        \end{equation}
        where we denote~$\kappa(\varphi)=\int_{\R^d}\varphi\,\d\kappa$ and~$\nu(\psi)=\int_{\cX}\psi\,\d\nu$.}
        \item{There exist constants~$c_1,c_2>0$ and~$0\leq c_3\leq 1$ such that the following bounds hold pointwise:
        \begin{equation}
        \label{eq:Vreg}
        \|\nabla ^2 V\|^2_{\mathrm{HS}} \leq c_1^2\left(d+|\nabla V|^2\right),\qquad \Delta V\leq c_2 d + \frac{c_3 \beta}{2}|\nabla V|^2.
        \end{equation}
        }
        \item{The function~$U$ is even, and for all multi-indices~$\alpha_1,\alpha_2,\alpha_3 \in\N^d$ such that~$|\alpha_i|\leq i$ for~${1\leq i\leq 3}$, it holds
        \begin{equation}
            \label{eq:Ureg}
        \partial^{\alpha_3} U\in L^2(\kappa),\qquad (\partial^{\alpha_2} U)(\partial^{\alpha_1}U) \in L^2(\kappa).
        \end{equation}}
        \end{itemize}
    \end{hypothesis}
    All the conditions on~$V$ and~$U$ (apart from the parity condition on~$U$) imposed by Assumptions~\ref{hyp:smoothness} and~\ref{hyp:UVreg} are satisfied whenever these two functions are smooth and grow quadratically at infinity.
    In particular, growth and integrability conditions on~$V$,~$D^{\pm 1}$ and their derivatives are vacuous whenever~$\cX$ is compact.

    The final assumption, which will be used to prove Corollary~\ref{corr:pathwise}, restricts the dynamics~\eqref{eq:underdamped_friction_dependent} to the standard physical setting, in which the kinetic energy is a positive non-degenerate quadratic form.
    \begin{hypothesis}
        \label{hyp:phys_kinetic}
        The function~$U$ in~\eqref{eq:underdamped_friction_dependent} is given by
        \begin{equation}
            \label{eq:phys_kinetic}
            U(p) = \frac12p^\top M^{-1}p
        \end{equation}
        for some constant matrix~$M\in\cSpd$.
    \end{hypothesis}

Our assumptions~(in particular Assumptions~\ref{hyp:smoothness} and~\ref{hyp:elliptic}), ensure that both~\eqref{eq:underdamped_friction_dependent} and~\eqref{eq:overdamped_friction_dependent} admit unique global-in-time solutions, see Lemma~\ref{lemma:wp} below.

\paragraph{Overdamped approximation.}
Our main result gives two convergence properties for the time-rescaled position process~$X^\lambda$, in the limit of a large friction intensity~$\lambda\to +\infty$.\\
Recall that~${(X_0=q_0^\lambda,p_0^\lambda)\sim \mu_0}$ for any~$\lambda>0$.
\begin{theorem}
\label{thm:overdamped_limit}
Assume that~$\mu_0\ll \mu$ is such that
\begin{equation}
\label{eq:warm_start}
(X_0,p_0)\sim \mu_0,\qquad \frac{\d\mu_0}{\d\mu}\in L^p(\mu)\qquad\text{for some }p\in(1,\infty].
\end{equation}
Suppose that Assumptions~\ref{hyp:smoothness},~\ref{hyp:hypoellipticity},~\ref{hyp:elliptic} and \ref{hyp:UVreg} are satisfied, and let~$q\in[1,\infty)$ be the H\"older-conjugate of~$p$, namely~$q = p/(p-1)$.

Then for any~$T>0$ and~$0<\alpha\leq\frac{2}{q}$, there exists~$C(T,\alpha,\mu_0)>0$ such that    
    \begin{equation}
    \label{eq:convergence_in_l2}
    \underset{0\leq t\leq T}{\sup}\,\E_{\mu_0}\left[|X_t^\lambda-X_t|^\alpha\right]\leq \frac{C(T,\alpha,\mu_0)}{\lambda^{\alpha/2}}.
    \end{equation}
\end{theorem}
The condition~\eqref{eq:warm_start} on the initial data corresponds to a ``warm-start'' assumption, which could be relaxed to allow for Dirac initial conditions, at the cost of a dedicated analysis for small times, which will not be our concern here.
The convergence in mean estimate~\eqref{eq:convergence_in_l2} trivially implies the following bound in Wasserstein~$\cW_\alpha$-distance between the time-marginals of~$X^\lambda$ and those of~$X$ for any~${\alpha\in (0,2/q]}$:
\begin{equation}
    \label{eq:wasserstein_bound}
    \underset{0\leq t\leq T}{\sup}\, \cW_\alpha\left(\mathrm{Law}_{\P_{\mu_0}}(X_t^\lambda),\mathrm{Law}_{\P_{\mu_0}}(X_t)\right)\leq \frac{C'(T,\alpha,\mu_0)}{\sqrt{\lambda}}
\end{equation}
for some~$C'(T,\alpha,\mu_0)>0$ independent of~$\lambda$.

The following corollary upgrades the approximation of Theorem~\ref{thm:overdamped_limit} to a pathwise convergence property. We denote by~$X^\lambda_*\P_{\mu_0}$ and~$X_*\P_{\mu_0}$ the path distributions of~$X^\lambda$ and~$X$ respectively, namely the pushforward measures of~$\P_{\mu_0}$ by the maps~$X^\lambda$ and~$X$, which are elements of~$\cP\left(\mathcal C([0,T];\cX)\right)$.
\begin{corollary}
    \label{corr:pathwise}
 Assume that the hypotheses of Theorem~\ref{thm:overdamped_limit} hold, and additionally that Assumption~\ref{hyp:phys_kinetic} is satisfied. Then the following convergence of path distributions holds:
    \begin{equation}
    \label{eq:convergence_in_law}
    X^\lambda_* \P_{\mu_0}\xrightarrow[\lambda\to +\infty]{\rm{weakly}}X_*\P_{\mu_0}\in\mathcal P(\mathcal C([0,T];\cX)).
\end{equation}
\end{corollary}
We finally give a straightforward corollary of Theorem~\ref{thm:overdamped_limit}, concerning the convergence of trajectory averages on finite time-invervals. Such averages are commonplace in sampling methods using the dynamics~\eqref{eq:underdamped_friction_dependent} or~\eqref{eq:overdamped_friction_dependent} to estimate expectations under Gibbs measure~$\nu$.
\begin{corollary}
    \label{corr:traj_averages}
    Assume that the hypotheses of Theorem~\ref{thm:overdamped_limit} hold with~$p\geq 2$, and let~$\eta\in\left[\frac{p}{2(p-1)},1\right]$, $r\in\left[1,\frac{2(p-1)}{p}\eta\right]$.
    
    Let~$T>0$ and~$\varphi\in \mathcal C^{0,\eta}(\cX)$, then it holds
    \begin{equation}
        \E_{\mu_0}\left[\left|\frac1T\int_0^T \varphi(X_s^\lambda)\,\d s - \frac1T\int_0^T \varphi(X_s)\,\d s\right|^r\right] = \bigo\left(\lambda^{-r/(2\eta)}\right) = \bigo\left(\lambda^{-1/2}\right)
    \end{equation}
    in the limit~$\lambda\to +\infty$.
\end{corollary}

Theorem~\ref{thm:overdamped_limit}, as well as Corollaries~\ref{corr:pathwise} and~\ref{corr:traj_averages} will be proven in Section~\ref{sec:proof_main}.

The key technical tool in the proof of Theorem~\ref{thm:overdamped_limit} is the following result.
\begin{lemma}[Uniform hypocoercivity in~$\wLmu$]
    \label{lemma:hypocoercivity}
    Let Assumptions~\ref{hyp:elliptic} and~\ref{hyp:UVreg} be satisfied.
    \begin{itemize}
    \item{For any~$\lambda>0$, the generator~$\cL_\lambda$ is invertible on~$L_0^2(\mu)$. Namely, for any~$\varphi\in\wLmuz$, the Poisson equation
    \begin{equation}
        \cL_\lambda f = \varphi
    \end{equation}
    has a unique solution~$f=\cL_\lambda^{-1}\varphi\in\wLmuz$.}
    \item{If~$\varphi\in (\mathrm{Id}-\Pi_0)L_0^2(\mu)$, it furthermore holds that for some constant~$C>0$,
    \begin{equation}
        \label{eq:hypocoercivity}
       \forall\,\lambda\geq 1,\qquad \|\cL_\lambda^{-1}\varphi\|_{L^2(\mu)}\leq C\|\varphi\|_{L^2(\mu)},\qquad \|\nabla_p \cL_\lambda^{-1}\varphi\|_{L^2(\mu)} \leq \frac{C}{\sqrt\lambda}\|\varphi\|_{L^2(\mu)}.
    \end{equation}
    }
    \end{itemize}
\end{lemma}
The main interest of this result, beyond confirming that the dynamics~\eqref{eq:underdamped_friction_dependent} indeed satisfies the~$L^2$-hypocoercivity property, is the uniformity of the estimates in~\eqref{eq:hypocoercivity}.
Whereas standard hypocoercive estimates only give~$\|\cL_\lambda^{-1}\varphi\|_{\wLmu} = \bigo(\max(\lambda,\lambda^{-1}))\|\varphi\|_{\wLmu}$ for general~$\varphi\in\wLmuz$, this bound can be improved whenever~$\varphi\in \ker\Pi_0$. 
The proof of Lemma~\ref{lemma:hypocoercivity}, which follows the algebraic viewpoint of~\cite{BFLS22}, is given in Section~\ref{sec:proofs} below.

\section{Proofs}
\label{sec:proof_main}
We now prove our main results. To prove Theorem~\ref{thm:overdamped_limit}, we will establish the result for~$\mu_0=\mu$ and~$p=2$, and conclude for more general initial distributions~$\mu_0$ by a simple reweighting argument. This will in particular establish the weak convergence of all finite-dimensional time marginals of~$X^\lambda$ towards those of~$X$.
To prove Corollary~\ref{corr:pathwise}, we exploit the specific form of the dynamics under Assumption~\eqref{hyp:phys_kinetic} to obtain the tightness of the family of~$\left(X^\lambda_*\P_{\mu_0}\right)_{\lambda\geq 1}$ in~$\mathcal P(\mathcal C([0,T];\cX))$, which will conclude the argument by a classical corollary of Prokhorov's theorem (see~\cite[Theorem 8.1]{B68}).

The first step is a simple computation with It\^o's formula, from which the ``noise-induced drift'' term appears naturally.
\paragraph{Proof of Theorem~\ref{thm:overdamped_limit}: a first It\^o computation.}
For any~$\lambda>0$, since the quadratic covariation $\left\langle q^\lambda\right\rangle$ in the position variable vanishes, it holds
\begin{equation}
    \d\left[D(q_t^\lambda)p_t^\lambda\right] = -D(q_t^\lambda)\nabla V(q_t^\lambda)\,\d t - \lambda\,\d q_t^\lambda + \sqrt{\frac{2\lambda}{\beta}}D^{1/2}(q_t^\lambda)\,\d W_t^\lambda + D'(q_t^\lambda)\left[\nabla U(p_t^\lambda)\right]p_t^\lambda\d t,
\end{equation}
where we recall the definition~\eqref{eq:diff_notation} of~$D'$. In particular, for any~${1\leq k\leq d}$,
\begin{equation}
    \label{eq:drift_term}
    \left(D'(q)\left[\nabla U(p)\right]p \right)_k = \sum_{i,j=1}^d \partial_i D_{kj}(q)\partial_i U(p)p_j.
\end{equation}

Isolating~$\d q_t^\lambda$ and integrating over~$[0,\lambda t]$, we get:
    \begin{equation}
        \label{eq:ipp_sto}
        \begin{aligned}
        q_{\lambda t}^\lambda- q_0^\lambda &= -\frac{1}{\lambda}\int_0^{\lambda t}D(q_s^\lambda)\nabla V(q_s^\lambda)\,\d s+ \sqrt{\frac{2}{\lambda\beta}}\int_0^{\lambda t}D^{1/2}(q_s^\lambda)\,\d W_s^\lambda\\
        &+ \frac{1}{\lambda}\int_0^{\lambda t}D'(q_s^\lambda)\left[\nabla U(p_s^\lambda)\right]p_s^\lambda\,\d s+\frac{D(q_0^\lambda)p_0^\lambda - D(q_{\lambda t}^\lambda)p_{\lambda t}^\lambda}{\lambda}.
        \end{aligned}
    \end{equation}

Defining~$X^\lambda_t = q_{\lambda t}$, this time-rescaling changes~\eqref{eq:ipp_sto} into
\begin{equation}
    \label{eq:perturbed_overdamped}
    X_t^\lambda = X_0 -\int_0^t \left[D(X_s^\lambda)\nabla V(X_s^\lambda)-\frac1\beta\,\mathrm{div} D(X_s^\lambda)\right]\,\d s + \sqrt{\frac{2}{\beta}}\int_0^t D^{1/2}(X_s^\lambda)\,\d W_s + R(t,\lambda),
\end{equation}
where we introduced the remainder term
\begin{equation}
    \label{eq:remainder}
    R(t,\lambda) = \frac{D(q_0^\lambda)p_0^\lambda-D(q_{\lambda t}^\lambda)p_{\lambda t}^\lambda}{\lambda} + \frac1\lambda\int_0^{\lambda t} \psi(q_s^\lambda,p_{s}^\lambda)\,\d s,\qquad \psi(q,p) = D'(q)\left[\nabla U(p)\right] p - \frac1\beta \mathrm{div} D(q),
\end{equation}
and use~$X_0^\lambda=q_0^\lambda=X_0$.

Since~$p\in L^2(\kappa)$,~$U\in H^1(\kappa)$ and~$D\in W^{1,\infty}(\kappa)$ (from Assumptions~\ref{hyp:elliptic} and~\ref{hyp:UVreg}), an integration by parts shows that component-wise:
\begin{equation}
    \Pi_0 \left(D'(q)\left[\nabla U(p)\right] p \right) = \frac1\beta\mathrm{div}D(q),
\end{equation}
so that~$\psi \in \left((\Id-\Pi_0)\wLmu\right)^d \subset \wLmuz^d$. Indeed, from~\eqref{eq:diff_notation} and the form~\eqref{eq:gibbs} for~$\kappa$, for any~${1\leq k\leq d}$,
\begin{equation}
    \label{eq:ipp_fonda}
    \begin{aligned}
    \int_{\R^d}\left(D'(q)\left[\nabla U(p)\right] p \right)_k\,\kappa(\d p) &= \sum_{i,j=1}^d \int_{R^d}\partial_{i}D_{kj}(q)(\partial_i U(p))p_j\,\kappa(\d p)\\
    &=\frac{1}{\beta}\sum_{i,j=1}^d\partial_i D_{kj}(q)\int_{\R^d}(\partial_i p_j)\,\kappa(\d p)\\
    &=\frac{1}{\beta}\sum_{i,j=1}^d \delta_{ij}\partial_i D_{kj}(q)\\
    &=\frac{1}{\beta}\left(\div\,D(q)\right)_k.
    \end{aligned}
\end{equation}
In view of the form~\eqref{eq:overdamped_friction_dependent} of the limiting equation, we interpret~$R(\cdot,\lambda)$ as a perturbation term.

Using the first item in Lemma~\ref{lemma:hypocoercivity}, we introduce, for each~$\lambda\geq 1$, the solution $\Phi_\lambda\in \wLmu^d$ to the (component-wise) Poisson equation:
\begin{equation}
    \label{eq:poisson_equation}
    \cL_\lambda\Phi_{\lambda} = \psi.
\end{equation}
Since~$\psi$ is smooth and~$\cL_\lambda$ is hypoelliptic under Assumption~\eqref{hyp:hypoellipticity},~$\Phi_\lambda$ is smooth, and we may write the integral in~\eqref{eq:remainder} as 
\begin{equation}
    \label{eq:remainder_integral}
    \frac1\lambda\int_0^{\lambda t} \psi(q_s^\lambda,p_{ s}^\lambda)\,\d s = \frac{\Phi_\lambda(X_t^\lambda,p_{\lambda t}^\lambda) - \Phi_\lambda(X_0,p_0^\lambda)}{\lambda} - \sqrt{\frac{2}{\beta}}\int_0^t \nabla_p \Phi_{\lambda}(X_s^\lambda,p_{\lambda s}^\lambda)^\top D(X_s^\lambda)^{-1/2}\,\d W_s,
\end{equation}
using It\^o's formula and making the same rescaling in time as the one leading to~\eqref{eq:perturbed_overdamped}.
Using these estimates, the proof of the first item~\eqref{eq:convergence_in_l2} in Theorem~\ref{thm:overdamped_limit} follows from a Gr\"onwall-type argument, which we now detail.

\paragraph{Proof of Theorem: proof of~\eqref{eq:convergence_in_l2}.}
Fix~$T>0$ and assume as a first step that~$\mu_0=\mu$. By H\"older's inequality, it suffices to show the estimate~\eqref{eq:convergence_in_l2} for~$\alpha=2$. By stationarity,~$X_t,X_t^\lambda\sim \nu$ and~$p_{\lambda t}^\lambda\sim\kappa$ for all~$0\leq t\leq T$ under~$\P_\mu$.
Note also that, since~$\nabla_p\Phi_\lambda$ is in~$\wLmu$ component-wise from the second item in Lemma~\ref{lemma:hypocoercivity}, the local martingale on the right-hand side of~\eqref{eq:remainder_integral} is an~$L^2(\P_\mu)$ $W$-martingale.

We denote the difference process by
\begin{equation}
    E_t^\lambda := X_t^\lambda - X_t,
\end{equation}
which has finite second-moments by the first item in Assumption~\ref{hyp:UVreg} under~$\P_\mu$, and write, for~${0\leq t \leq T}$,
\begin{equation}
    \label{eq:young_inequality}
        |E_t^\lambda|^2 \leq 3 \left[\left|\int_0^t \left[b(X_s^\lambda)-b(X_s)\right]\,\d s\right|^2+\left|\int_0^t \left[\sigma(X_s^\lambda)-\sigma(X_s)\right]\,\d W_s\right|^2 + \left|R(t,\lambda)\right|^2\right],
\end{equation}
where~$b(x) = -D(x)\nabla V(x) + \frac1\beta\div\, D(x)$ and~$\sigma(x) = \sqrt{\frac2\beta}D(x)^{1/2}$. We estimate each of these terms separately. Let~$L_b,L_\sigma>0$ be Lipschitz constants for~$b$ and~$\sigma$ (in the norm~$\left\|\cdot\right\|_{\mathrm{HS}}$ for $\sigma$), which exist by Assumptions~\ref{hyp:smoothness} and~\ref{hyp:elliptic}.
Then the first two terms inside the bracket on the right side of~\eqref{eq:young_inequality} are bounded in expectation by
\begin{equation}
    \label{eq:lipschitz_bound}
    (tL_b^2+L_\sigma^2)\int_0^t \E_\mu\left[|E_s^\lambda|^2\right]\,\d s,
\end{equation}
 using It\^o's isometry to bound the second term. To control the remainder, we write, using~\eqref{eq:remainder} and~\eqref{eq:remainder_integral}
 \begin{equation}
    \begin{aligned}
    \E_\mu\left[\left|R(t,\lambda)\right|^2\right] &\leq 5\left(\frac{2}{\lambda^2}\left[\|D p\|^2_{\wLmu}+\|\Phi_\lambda\|^2_{\wLmu}\right] + \frac{2}{\beta}\E_\mu\left[\left|\int_0^t \nabla_p \Phi_\lambda(X_s^\lambda,p_{\lambda s}^\lambda)^\top D(X_s^\lambda)^{-1/2}\,\d W_s\right|^2\right]\right)\\
    &\leq 5\left(\frac{2}{\lambda^2}\left[\|D p\|^2_{\wLmu}+\|\Phi_\lambda\|^2_{\wLmu}\right] + \frac{2T}{\beta}\|D^{-1}\|_{L^\infty(\cX;\cSpd)}\|\nabla_p\Phi_\lambda\|^2_{\wLmu,\mathrm{HS}}\right),
    \end{aligned}
 \end{equation}
using once again It\^o's isometry in the second line. We may now use the two estimates of Lemma~\ref{lemma:hypocoercivity} to obtain
\begin{equation}
    \label{eq:l2_bounds}
    \E_\mu\left[\left|R(\lambda,t)\right|^2\right] \leq \frac{C_T}{\lambda}
\end{equation}
for some constant~$C_T>0$ independent of~$\lambda$. Collecting the estimates~\eqref{eq:young_inequality},~\eqref{eq:lipschitz_bound} and~\eqref{eq:l2_bounds}, we get, taking expectations,
\begin{equation}
    \E_\mu\left[|E_t^\lambda|^2\right] \leq 3\left(TL_b^2+L_\sigma^2\right)\int_0^t\E_\mu\left[|E_s^\lambda|^2\right]\,\d s + \frac{3C_T}{\lambda},
\end{equation}
from which Gr\"onwall's lemma gives
\begin{equation}
    \underset{0\leq t\leq T}{\sup}\,\E_\mu\left[|E_t^\lambda|^2\right] \leq \frac{3C_T}{\lambda}\e^{3T(TL_b^2+L_\sigma^2)},
\end{equation}
which concludes the proof of~\eqref{eq:convergence_in_l2} for~$\mu_0=\mu$ and~$\alpha=2$.

To extend the argument to~$\mu_0$, we simply reweight the initial condition. Let~$0<\alpha\leq 2/q$. We write, by H\"older's inequality,
\begin{equation}
    \label{eq:conditioning_argument}
    \begin{aligned}
    \E_{\mu_0}\left[|E_t^\lambda|^\alpha\right] =\E_{\mu}\left[|E_t^\lambda|^\alpha \frac{\d\mu_0}{\d \mu}(X_0,p_0)\right]\leq \E_\mu\left[|E_t^\lambda|^{q\alpha}\right]^{1/q}\left\|\frac{\d \mu_0}{\d \mu}\right\|_{L^p(\mu)}\leq \E_\mu\left[|E_t^\lambda|^2\right]^{\alpha/2}\left\|\frac{\d \mu_0}{\d \mu}\right\|_{L^p(\mu)},
    \end{aligned}
\end{equation}
from which the proof of~\eqref{eq:convergence_in_l2} follows from the case~$(\mu_0,\alpha)=(\mu,2)$.

\paragraph{Proof of Corollary~\ref{corr:pathwise}: tightness.}
We now show that the pathwise convergence~\eqref{eq:convergence_in_law} holds under the further Assumption~\ref{hyp:phys_kinetic}.
Since all finite-dimensional marginals of~$X^\lambda$ converge weakly to those of~$X$ by the estimate~\eqref{eq:convergence_in_l2}, it is sufficient, in order to conclude, to prove tightness for the family of pushforward measures~$(X^\lambda_* \P_{\mu_0})_{\lambda\geq 1}$ on~$\mathcal C([0,T];\cX)$.
We use the following criterion, which is an immediate corollary of~\cite[Theorem 12.3]{B68}. If the family of initial laws~$\left(\left[X_0^\lambda\right]_*\P_{\mu_0}\right)_{\lambda\geq 1}$ is tight in~$\cP(\cX)$, and there exist real numbers~$a,b,C>0$ such that
\begin{equation}
    \label{eq:tightness_criterion}
    \forall\,\lambda>0,\,0\leq s<t\leq T,\qquad \E_{\mu_0}\left[\left|X_s^\lambda-X_t^\lambda\right|^a\right]\leq C(t-s)^{1+b},
\end{equation}
then~$(X^\lambda_*\P_{\mu_0})_{\lambda\geq 1}$ is tight in~$\cP(\cC([0,T];\cX))$. Since~$\left[X_0^\lambda\right]_*\P_{\mu_0}=\mu_0$ for all~$\lambda$, it is sufficient to control moments of the form~\eqref{eq:tightness_criterion}.

Fix~$0\leq s< t \leq T$. The first step is to write the velocity~$M^{-1}p_t^\lambda$ in a more explicit form. We recall that under Assumption~\ref{hyp:phys_kinetic},~$U(p)=\frac12 p^\top M^{-1}p$ for some constant~$M\in\cSpd$. We view, in the spirit of~\cite{BHVW17,WSW24}, the equation
\begin{equation}
    \begin{aligned}
        \d(M^{-1/2}p^\lambda)_t = &-M^{-1/2}\nabla V(q_t^\lambda)\,\d t - \lambda\left(M^{-1/2}D(q_t^\lambda)^{-1}M^{-1/2}\right)M^{-1/2}p_t^\lambda\,\d t \\
        &+ \sqrt{\frac{2\lambda}{\beta}}M^{-1/2}D(q_t^\lambda)^{-1/2}\,\d W_t 
    \end{aligned}
\end{equation}
as a time-dependent linear ODE with a source term.
To solve this equation, we introduce the matrix of fundamental solutions to the associated homogeneous problem, namely the~$\cSpd$-valued process defined pathwise by the ODE:
\begin{equation}
    \label{eq:fundamental_matrix_ode}
    \d\cR^\lambda_t = -\lambda M^{-1/2} D(q_t^\lambda)^{-1}M^{-1/2}\cR^\lambda_t\,\d t,\qquad \cR^\lambda_0 = \Id.
\end{equation}
We will use the shorthand~$D_M$ for the matrix field appearing in~\eqref{eq:fundamental_matrix_ode}, i.e.
\begin{equation}
D_M(q) = M^{-1/2}D(q)^{-1}M^{-1/2}.
\end{equation}
Note that~$D_M$ also satisfies Assumption~\eqref{hyp:elliptic}, so that the existence, uniqueness and~$W$-adaptiveness of~$\cR_\lambda$ on~$[0,\lambda T]$ all follow from the general theory of linear ODEs with continuous coefficients, see for instance~\cite[Section 3.4]{T00}.

Using Duhamel's principle, we may write
\begin{equation}
    \begin{aligned}
        M^{-1/2}p^\lambda_t = &\cR_{t}^\lambda M^{-1/2}p_0^\lambda - \cR_t^\lambda\int_0^t (\cR_{s}^\lambda)^{-1}M^{-1/2}\nabla V(q_s^\lambda)\,\d s\\
        &+\sqrt{\frac{2\lambda}{\beta}}\cR_t^\lambda\int_0^t (\cR_{s}^\lambda)^{-1} M^{-1/2}D(q_s^\lambda)^{-1/2}\,\d W^\lambda_s.
    \end{aligned}
\end{equation}
The validity of this expression may be checked a posteriori by applying Itô's product rule to~$(\cR_t^\lambda)^{-1}M^{-1/2}p_t^\lambda$, and using the expression
\begin{equation}
    \label{eq:matrix_diff_expr}
    \d\left(\cR_t^\lambda\right)^{-1} = \lambda \left(\cR_t^\lambda\right)^{-1}D_M(q_t^\lambda)\,\d t,
\end{equation}
which itself is a consequence the well-known inverse matrix derivative identity ${\partial_t (A(t)^{-1}) = - A(t)^{-1}\partial_t A(t) A(t)^{-1}}$.

Left-multiplying~$M^{-1/2}p_t$ by~$M^{-1/2}$ and integrating over~$[\lambda s,\lambda t]$, it follows, according to the SDE~\eqref{eq:underdamped_friction_dependent}, that the time-rescaled position increments can be written as
    \begin{equation}
        \begin{aligned}
        X_t^\lambda - X_s^\lambda = \int_{\lambda s}^{\lambda t}M^{-1}p_r^\lambda\,\d r = M^{-1/2}\left[\int_{\lambda s}^{\lambda t}\left(\cR_{r}^\lambda M^{-1/2}p_0-\cR_r^\lambda\int_0^r (\cR_{u}^\lambda)^{-1}M^{-1/2}\nabla V(q_u^\lambda)\,\d u\right.\right.\\\left.\left.+\sqrt{\frac{2\lambda}{\beta}}\cR^\lambda_r\int_0^r \left(\cR_{u}^{\lambda}\right)^{-1}M^{-1/2} D(q_u^\lambda)^{-1/2}\,\d W^\lambda_u\right)\,\d r\right],
    \end{aligned}
\end{equation}
which we split into three summands as~$M^{-1/2}\left[I_1(s,t)+I_2(s,t)+I_3(s,t)\right]$.

Take~$\gamma> 1$. Using the bound~$|X_t^\lambda-X_s^\lambda|^\gamma\leq \|M^{-1/2}\|_{\mathrm{op}}^\gamma 3^{\gamma-1}(|I_1(s,t)|^\gamma + |I_2(s,t)|^\gamma + |I_3(s,t)|^\gamma)$, it remains to control the three integral terms separately, seeking uniform-in-$\lambda$ estimates for their~$\P_{\mu_0}$-expectation.
We use the inequality:
\begin{equation}
    \label{eq:fundamental_matrix_bound}
    \forall\,0\leq s\leq t\leq T,\qquad \left\|\cR_t^\lambda\left(\cR_{s}^\lambda\right)^{-1}\right\|\leq \e^{-\lambda\varepsilon_{D,M}(t-s)}\qquad \P_{\mu_0}\text{-almost surely},
\end{equation}
where~$\|\cdot\|$ is the standard Euclidean operator norm on~$\R^d$, for some~$\varepsilon_{D,M}>0$ independent of~$\lambda$.
This identity follows from the uniform bound~\eqref{hyp:elliptic}, upon applying Gr\"onwall's inequality pathwise, to~$|v(t)|^2 = \left|\cR_t^\lambda (\cR_s^\lambda)^{-1}v_0\right|^2$, where~$v$ solves the ODE:
\begin{equation}
\frac{\d}{\d t}v(t) = -\lambda D_M (q_t^\lambda)v(t),\qquad v(s) = v_0,
\end{equation}
for any~$v_0\in\R^d$ (see also~\cite[Problem 3.31]{T00}).

Applying~\eqref{eq:fundamental_matrix_bound} to~$I_1$, we get the following almost sure inequality:
\begin{equation}
    \left|I_1(s,t)\right|^\gamma \leq |M^{-1/2}p_0|^\gamma\left(\int_{\lambda s}^{\lambda t} \e^{-\lambda\varepsilon_{D,M} r}\,\d r\right)^\gamma,
\end{equation}
which gives, upon taking the expectation with respect to~$\mu_0$,
\begin{equation}
    \label{eq:i1_estimate}
    \begin{aligned}
    \E_{\mu_0}\left[\left|I_1(s,t)\right|^\gamma\right] &\leq\frac{\mu_0(|M^{-1/2}p_0|^\gamma)}{\lambda^\gamma \varepsilon_{D,M}^\gamma}[\e^{-\lambda^2\varepsilon_{D,M} s}-\e^{-\lambda^2\varepsilon_{D,M} t}]^\gamma\\
    &= \frac{\mu_0(|M^{-1/2}p_0|^\gamma)}{\lambda^\gamma \varepsilon_{D,M}^\gamma} \e^{-\lambda^2\gamma \varepsilon_{D,M} s}\left(1-\e^{-\lambda^2\varepsilon_{D,M}(t-s)}\right)^\gamma \\
    &\leq \frac{\mu_0(|M^{-1/2}p_0|^\gamma)}{\lambda^\gamma \varepsilon_{D,M}^\gamma} \sqrt{\lambda^2\varepsilon_{D,M}(t-s)}^\gamma \\
    &\leq \frac{\mu(|M^{-1/2} p_0|^{\gamma q})^{1/q}}{\varepsilon_{D,M}^{\gamma/2}}C_p(\mu_0)(t-s)^{\gamma/2},
    \end{aligned}
\end{equation}
using~$1-\e^{-x}\leq \min(1,x)\leq \sqrt x$ for~$x\geq 0$ in the third line, and using the reweighting argument of~\eqref{eq:conditioning_argument} in the last inequality, where we henceforth denote~${C_p(\mu_0) = \left\|\d\mu_0/\d\mu\right\|_{L^p(\mu)}}$. Since~$p_0$ is Gaussian under~$\mu$, the~$\gamma q$-th moment~$\mu(|M^{-1/2}p_0|^{\gamma q})$ is finite, and the estimate~\eqref{eq:i1_estimate} controls the contribution of~$I_1(s,t)$ to~$\E_{\mu_0}[|X_s^\lambda-X_t^\lambda|^\gamma]$ uniformly with respect to~$\lambda$.

Next, we treat the second term, writing:
\begin{equation}
    \begin{aligned}
    |I_2(s,t)|^\gamma &\leq \left(\int_{\lambda s}^{\lambda t}\int_{0}^r \left|\cR_r^\lambda(\cR_{u}^\lambda)^{-1} M^{-1/2}\nabla V(q_u^\lambda)\right|\,\d u\,\d r\right)^\gamma\\
    &\leq (\lambda(t-s))^{\gamma-1}\int_{\lambda s}^{\lambda t}\left(\int_0^r \e^{-\lambda\varepsilon_{D,M}(r-u)}|M^{-1/2}\nabla V(q_u^\lambda)|\,\d u\right)^\gamma\,\d r\\
    &\leq (\lambda(t-s))^{\gamma-1}\int_{\lambda s}^{\lambda t}\left(\int_{0}^r \e^{-\lambda\varepsilon_{D,M}(r-u)}\,\d u\right)^{\gamma-1}\left(\int_{0}^r \e^{-\lambda\varepsilon(r-u)}|M^{-1/2}\nabla V(q_u^\lambda)|^\gamma\,\d u\right)\,\d r\\
    &\leq \frac{(t-s)^{\gamma-1}}{\varepsilon_{D,M}^{\gamma-1}}\int_{\lambda s}^{\lambda t}\int_{0}^r \e^{-\lambda\varepsilon_{D,M}(r-u)}|M^{-1/2}\nabla V(q_u^\lambda)|^\gamma\,\d u\,\d r,
    \end{aligned}
\end{equation}
where we used Hölder's inequality twice: once in~$r$ in the second line for the Lebesgue measure, and once in~$u$ in the third line for the exponentially-weighted measure~$\e^{-\lambda\varepsilon_{D,M}(r-u)}\,\d u$.
Taking the expectation with respect to~$\P_{\mu_0}$, we get, with the same reweighting argument as in previous estimates:
\begin{equation}
    \label{eq:i2_estimate}
    \begin{aligned}
    \E_{\mu_0}\left[|I_2(s,t)|^\gamma\right] &\leq \underset{0\leq u\leq t}{\sup}\,\E_{\mu_0}[|M^{-1/2}\nabla V(q_u^\lambda)|^\gamma]\frac{(t-s)^{\gamma-1}}{\varepsilon_{D,M}^{\gamma-1}}\int_{\lambda s}^{\lambda t}\int_0^r \e^{-\lambda \varepsilon_{D,M}(r-u)}\,\d u\,\d r\\
    &\leq C_p(\mu_0)\|M^{-1/2}\|_{\mathrm{op}}^\gamma\mu\left(|\nabla V|^{\gamma q}\right)^{1/q}\frac{(t-s)^{\gamma-1}}{\varepsilon_{D,M}^{\gamma-1}}\frac{\lambda t-\lambda s}{\lambda \varepsilon_{D,M}}\\
    &= C_p(\mu_0)\|M^{-1/2}\|_{\mathrm{op}}^\gamma\mu\left(|\nabla V|^{\gamma q}\right)^{1/q}\frac{(t-s)^\gamma}{\varepsilon_{D,M}^\gamma}.
    \end{aligned}
\end{equation}

We finally treat the last term~$I_3(s,t)$, which requires some care. In the proof of~\cite[Lemma~3.1]{WSW24}, the authors suggest controlling similar double integrals by a (stochastic) Fubini's theorem followed by Doob's maximal inequality. This argument, while giving sufficient control formally, seems incorrect, as the resulting time-swapped integrands are not~$W$-adapted, nor in general (sub)martingales. A similar mistake can be found in the proof of~\cite[Section 3.4.2, Lemma 10]{H13}.
To avoid this difficulty, we write instead
\begin{equation}
    \label{eq:i3_rewrite}
    I_3(s,t) = \int_{\lambda s}^{\lambda t}Z_r^\lambda\,\d r,\qquad Z_t^\lambda := \sqrt{\frac{2\lambda}{\beta}}\cR_t^\lambda \int_0^t \left(\cR_s^\lambda\right)^{-1}M^{-1/2}D(q_s^\lambda)^{-1/2}\,\d W^\lambda_s,
\end{equation}
from which we get, using~\eqref{eq:fundamental_matrix_ode},
\begin{equation}
    \label{eq:Z_sde}
    \d Z_t^\lambda = -\lambda D_M(q_t^\lambda)Z_t^\lambda\d t + \sqrt{\frac{2\lambda}{\beta}}M^{-1/2}D(q_t^\lambda)^{-1/2}\,\d W_t^\lambda,
\end{equation}
against which we integrate~$-\lambda^{-1}D_M(q^\lambda)^{-1}$ on~$[\lambda s,\lambda t]$ to obtain
\begin{equation}
    \label{eq:I3_ito_decomposition}
    I_3(s,t) = \sqrt{\frac{2}{\lambda \beta}}\int_{\lambda s}^{\lambda t}\left[D_M^{-1}M^{-1/2}D^{-1/2}\right](q_r^\lambda)\,\d W^\lambda_r-\lambda^{-1}\int_{\lambda s}^{\lambda t}D_M(q_r^\lambda)^{-1}\d Z_r^\lambda.
\end{equation}
To treat the first term, we use the Burkholder--Davis--Gundy inequality to show the existence of~$C_{\gamma,\beta}>0$ such that
\begin{equation}
    \begin{aligned}
        \label{eq:I3_martingale_term}
        \E_{\mu_0}\left[\left|\sqrt{\frac{2}{\lambda\beta}}\int_{\lambda s}^{\lambda t}\left[D_M^{-1}M^{-1/2}D^{-1/2}\right](q_r^\lambda)\,\d W_r^\lambda\right|^\gamma\right] &\leq C_{\gamma,\beta}\lambda^{-\gamma/2} \E_{\mu_0}\left[\left(\int_{\lambda s}^{\lambda t}\left\|\left[D_M^{-1}M^{-1/2}D^{-1/2}\right](q_r^\lambda)\right\|^2_{\mathrm{HS}}\d r\right)^{\gamma/2}\right]\\
        &\leq C_{\gamma,\beta}\|D_M^{-1}M^{-1/2}D^{-1/2}\|_{L^\infty_{\mathrm{HS}}}^{\gamma}(t-s)^{\gamma/2}.
    \end{aligned}
\end{equation}
We rewrite the second term in~\eqref{eq:I3_ito_decomposition}, using It\^o's product rule, which gives
\begin{equation}
    \label{eq:I3_bv_term_ipp}
    \int_{\lambda s}^{\lambda t}D_M(q_r^\lambda)^{-1}\,\d Z_r^\lambda = D_M(q_{\lambda t}^\lambda)^{-1} Z_{\lambda t}^\lambda - D_M(q_{\lambda s}^\lambda)^{-1}Z_{\lambda s}^\lambda -\int_{\lambda s}^{\lambda t}\left(D_M^{-1}\right)'(q_r^\lambda)[M^{-1}p_r^\lambda]Z_r^\lambda\,\d r.
\end{equation}
Controlling these three terms requires some moment bounds on~$Z^\lambda$. The idea is to apply It\^o's formula to~$Z^\lambda$ with~$f(z)=|z|^\gamma$. In view of the SDE~\eqref{eq:Z_sde} satisfied by~$Z^\lambda$, it holds for~$\gamma>2$:
\begin{equation}
    \label{eq:diff_inequality}
    \begin{aligned}
    \frac{\d}{\d r}\E_{\mu_0}\left[|Z_r^\lambda|^\gamma\right] &= -\lambda\gamma\E_{\mu_0}\left[|Z_r^\lambda|^{\gamma-2}\left(Z_r^{\lambda}\right)^\top D_M(q_r^\lambda)Z_r^{\lambda}\right]\\
    &+\lambda\frac{\gamma(\gamma-2)}{\beta}\E_{\mu_0}\left[|Z_r^\lambda|^{\gamma-4}M^{-1/2}D(q_r^\lambda)^{-1}M^{-1/2}:Z_r^\lambda\otimes Z_r^\lambda\right]\\
    &+\lambda\frac{\gamma}{\beta}\E_{\mu_0}\left[|Z_r^\lambda|^{\gamma-2}\mathrm{Tr}\,\left(M^{-1/2}D(q_r^\lambda)^{-1}M^{-1/2}\right)\right]\\
    &\leq -\lambda C^{(1)}_\gamma\E_{\mu_0}\left[|Z_r^\lambda|^\gamma\right] + \lambda C^{(2)}_\gamma\E_{\mu_0}\left[|Z_r^\lambda|^{\gamma-2}\right]\\
    &\leq -\lambda C^{(3)}_\gamma\E_{\mu_0}\left[|Z_r^\lambda|^\gamma\right] +  \lambda C^{(4)}_\gamma,
    \end{aligned} 
\end{equation}
where the constants~$C_{\gamma}^{(i)} >0$,~$1\leq i\leq 4$ depend on~$D,\beta$ and~$M$, but may be chosen uniformly in~$\lambda$. Indeed, by optimizing~$\theta^{\gamma-2}-\alpha t^{\gamma}$ with respect to~$\theta$, we get the inequality
\begin{equation}
    \forall\,\theta\geq 0,\alpha>0,\gamma>2\qquad \theta^{\gamma-2}\leq \alpha \theta^\gamma +\alpha^{-\frac{\gamma-2}{2}}\frac{2}{\gamma}\left(\frac{\gamma-2}{\gamma}\right)^{\frac{\gamma-2}{2}},
\end{equation}
which we use in the last line of~\eqref{eq:diff_inequality} with~$\alpha<C_\gamma^{(1)}/C_\gamma^{(2)}$ to absorb~$\E_{\mu_0}\left[|Z_r^\lambda|^{\gamma-2}\right]$ into the dissipative term.

Solving the differential inequality with~$Z_0^\lambda=0$, we get:
\begin{equation}
    \label{eq:Z_moment_bound}
    \E_{\mu_0}\left[|Z_r^\lambda|^\gamma\right] \leq \lambda C^{(4)}_{\gamma}\frac{1-\e^{-r \lambda C^{(3)}_{\gamma}}}{\lambda C^{(3)}_{\gamma}} \leq \frac{C_\gamma^{(4)}}{C_\gamma^{(3)}}
\end{equation}
for all~$r\geq 0$. To make this argument rigorous, one should really apply It\^o's formula to the stopped process~$Z^\lambda_{t\land \tau^\lambda_K}$ with~$\tau^\lambda_K = \inf\{t\geq 0: |Z_t^\lambda|\geq K\}$ and~$f_K$ a bounded~$\mathcal C^2$ function satisfying~$f_K(z)\1_{\{|z|\leq K\}} = |z|^\gamma$. One can then obtain~$\E_{\mu_0}\left[|Z^\lambda_{r\land \tau_K^\lambda}|^\gamma\right]\leq C$ by the computation~\eqref{eq:diff_inequality} (for some constant~$C$ uniform in~$K$ and~$\lambda$), and pass to the limit~$K\to+\infty$.

We now use the bound~\eqref{eq:Z_moment_bound} to control~\eqref{eq:I3_bv_term_ipp}, starting with the boundary terms. Adding and substracting the term $D_M(q_{\lambda s}^\lambda)^{-1}Z_{\lambda t}^\lambda$ in the right-hand side of~\eqref{eq:I3_bv_term_ipp}, we estimate
\begin{equation}
    \label{eq:I3_telescoping_bound}
        |D_M(q_{\lambda t}^\lambda)^{-1} Z_{\lambda t}^\lambda - D_M(q_{\lambda s}^\lambda)^{-1}Z_{\lambda s}^\lambda|^\gamma\leq 2^{\gamma-1}\left( \|D_M^{-1}\|_{\cW_{\mathrm{op}}^{1,\infty}}^\gamma|q_{\lambda t}^\lambda-q_{\lambda s}^\lambda|^\gamma|Z_{\lambda s}^\lambda|^\gamma+\|D_M^{-1}\|_{L^\infty_{\mathrm{op}}}|Z_{\lambda t}^\lambda - Z_{\lambda s}^\lambda|^\gamma\right),
\end{equation}
where $\|D_M^{-1}\|_{\cW_{\mathrm{op}}^{1,\infty}}$ is a Lipschitz constant for~$D_M^{-1}$, which exists by Assumption~\ref{hyp:elliptic}. 
We first write, by a H\"older inequality and the SDE~\eqref{eq:underdamped_friction_dependent},
\begin{equation}
    |q_{\lambda t}^\lambda-q_{\lambda s}^\lambda|^\gamma \leq (\lambda(t-s))^{\gamma -1}\int_{\lambda s}^{\lambda t}| M^{-1}p_r^\lambda|^\gamma\,\d r,
\end{equation}
so that a Cauchy--Schwarz inequality gives, together with the moment bound~\eqref{eq:Z_moment_bound}:
\begin{equation}
    \begin{aligned}
        \label{eq:I3_telescoping_bound_q_incr}
        \E_{\mu_0}\left[|q_{\lambda t}^\lambda-q_{\lambda s}^\lambda|^\gamma|Z_{\lambda s}^\lambda|^\gamma\right]&\leq (\lambda(t-s))^{\gamma -1}\int_{\lambda s}^{\lambda t}\E_{\mu_0}\left[|M^{-1}p_r^\lambda|^\gamma|Z_{\lambda s}^\lambda|^\gamma\right]\,\d r\\
        &\leq \lambda^\gamma(t-s)^\gamma\left(\frac{C_{2\gamma}^{(4)}}{C_{2\gamma}^{(3)}}\right)^{1/2}\mu\left(\left|M^{-1}p\right|^{2\gamma q}\right)^{1/q}C_p(\mu_0),
    \end{aligned}
\end{equation}
where we used the usual reweighting argument in the final line.

Next, using the SDE~\eqref{eq:Z_sde} for~$Z^\lambda$, we write
\begin{equation}
    Z_{\lambda t}^\lambda - Z_{\lambda s}^\lambda = -\lambda\int_{\lambda s}^{\lambda t}D_M(q_r^\lambda) Z_r^\lambda\,\d r + \sqrt{\frac{2\lambda}{\beta}}\int_{\lambda s}^{\lambda t}M^{-1/2}D(q_r^\lambda)^{-1/2}\,\d W_r^\lambda.
\end{equation}
The martingale component is controlled with the Burkholder--Davis--Gundy inequality by
\begin{equation}
    \label{eq:I3_telescoping_bound_mart}
    \E_{\mu_0}\left[\left|\sqrt{\frac{2\lambda}{\beta}}\int_{\lambda s}^{\lambda t}M^{-1/2}D(q_t^\lambda)^{-1/2}\,\d W_t^\lambda\right|^\gamma\right]\leq C_{\gamma,\beta}\lambda^{\gamma}\left\|M^{-1/2}D^{-1/2}\right\|^\gamma_{L^\infty_\mathrm{HS}}(t-s)^{\gamma/2}.
\end{equation}

For the drift component, we estimate pathwise by a H\"older inequality:
\begin{equation}
    \left|-\lambda\int_{\lambda s}^{\lambda t}D_M(q_r^\lambda) Z_r^\lambda\,\d r\right|^\gamma \leq \lambda^\gamma \|D_M\|_{L^\infty_{\mathrm{op}}}^\gamma (\lambda(t-s))^{\gamma -1}\int_{\lambda s}^{\lambda t}|Z_r^\lambda|^\gamma\,\d r, 
\end{equation}
so that using the moment bound~\eqref{eq:Z_moment_bound} yields
\begin{equation}
    \label{eq:I3_telescoping_bound_drift}
    \E_{\mu_0}\left[\left|-\lambda\int_{\lambda s}^{\lambda t}D_M(q_r^\lambda) Z_r^\lambda\,\d r\right|^\gamma\right]\leq \lambda^{2\gamma}(t-s)^\gamma\|D_M\|_{L^\infty_{\mathrm{op}}}^\gamma \frac{C_\gamma^{(4)}}{C_\gamma^{(3)}}.
\end{equation}
Combining~\eqref{eq:I3_telescoping_bound_mart} and~\eqref{eq:I3_telescoping_bound_drift}, we get an estimate
\begin{equation}
    \label{eq:I3_Z_increment_transient}
    \E_{\mu_0}\left[|Z_{\lambda t}^\lambda -Z_{\lambda s}^\lambda|^\gamma\right] \leq C^{(5)}_{\gamma}(\lambda^{2\gamma}(t-s)^\gamma + \lambda^\gamma (t-s)^{\gamma/2})
\end{equation}
for some constant~$C^{(5)}_\gamma>0$ independent of $\lambda$. Because of the term scaling as~$\lambda^{2\gamma}$, this bound alone does not give sufficient control on the increments of~$Z^\lambda$ as~$\lambda\to+\infty$. For this we use the following argument.

From the moment bound~\eqref{eq:Z_moment_bound}, we simply have
\begin{equation}
    \E_{\mu_0}\left[|Z_{\lambda t}^\lambda -Z_{\lambda s}^\lambda|^\gamma\right] \leq 2^{\gamma}\frac{C_\gamma^{(4)}}{C_\gamma^{(3)}},
\end{equation}
which combined with~\eqref{eq:I3_Z_increment_transient} gives, for some constant~$C^{(6)}_\gamma>0$ independent of $\lambda$,
\begin{equation}
    \label{eq:I3_Z_increment_final_bound}
    \E_{\mu_0}\left[|Z_{\lambda t}^\lambda -Z_{\lambda s}^\lambda|^\gamma\right] \leq \min\left(C_\gamma^{(5)}\left[\lambda^{2\gamma}(t-s)^\gamma+\lambda^\gamma(t-s)^{\gamma/2}\right],C^{(6)}_\gamma\right)\leq \max(2C_\gamma^{(5)},C^{(6)}_\gamma)\lambda^{\gamma}(t-s)^{\gamma/2},
\end{equation}
where we used the elementary inequality $\min(u^2+u,M)\leq \max(2,M)u$ with~$u=\lambda^\gamma(t-s)^{\gamma/2}$, which can be easily verified by considering the cases~$0\leq u\leq 1$ and~$u>1$ separately.

Taking the expectation under~$\mu_0$ of~\eqref{eq:I3_telescoping_bound} and inserting the estimates~\eqref{eq:I3_telescoping_bound_q_incr},~\eqref{eq:I3_Z_increment_final_bound}, we obtain
\begin{equation}
    \label{eq:I3_boundary_terms_final_bound}
    \begin{aligned}
        \E_{\mu_0}\left[\left|\frac1\lambda\left(D_M(q_{\lambda t}^\lambda)^{-1}Z_{\lambda t}^\lambda - D_M(q_{\lambda s}^\lambda)^{-1}Z_{\lambda s}^\lambda\right)\right|^\gamma\right]&\leq C^{(7)}_\gamma\left((t-s)^\gamma+(t-s)^{\gamma/2}\right)\\
        &\leq C^{(7)}_\gamma(1+T^{\gamma/2})(t-s)^{\gamma/2},
    \end{aligned}
\end{equation}
for some constant~$C^{(7)}_\gamma$ independent of~$\lambda$, which gives the desired control on the contributions of the boundary terms in the right-hand-side of~\eqref{eq:I3_bv_term_ipp} to the second term on the right-hand-side of~\eqref{eq:I3_ito_decomposition}. In the last line of~\eqref{eq:I3_boundary_terms_final_bound}, we used $t-s\leq T$ to absorb~$(t-s)^{\gamma/2}$ in the prefactor.

It remains to estimate the contribution of the integral term on the right-hand-side of~\eqref{eq:I3_bv_term_ipp}. By  H\"older's inequality we get
\begin{equation}
    \label{eq:I3_bv_integral_term}
    \begin{aligned}
    \mathbb E_{\mu_0}\left[\left|\lambda^{-1}\int_{\lambda s}^{\lambda t}\left(D_M^{-1}\right)'(q_r^\lambda)\left[M^{-1}p_t^\lambda\right]Z_r^\lambda\,\d r\right|^\gamma\right]&\leq C_{\gamma}^{(8)}\lambda^{-1}(t-s)^{\gamma-1}\int_{\lambda s}^{\lambda t}\E_{\mu_0}\left[\left|M^{-1}p_r^\lambda\right|^\gamma\left|Z_r^\lambda\right|^\gamma\right]\,\d r\\
&\leq C^{(8)}_{\gamma}(t-s)^{\gamma}\left(\frac{C_{2\gamma}^{(4)}}{C_{2\gamma}^{(3)}}\right)^{1/2}\mu\left(\left|M^{-1}p\right|^{2\gamma q}\right)^{1/q}C_p(\mu_0)\\
&\leq C^{(9)}_{\gamma}T^{\gamma/2}(t-s)^{\gamma/2}
    \end{aligned}
\end{equation}
where~$C^{(8)}_{\gamma}=\left\|\left(D_M^{-1}\right)'\right\|^\gamma_{L^\infty(\cX;\cL(\R^d;\R^{d\times d}))}$ is finite by Assumptions~\ref{hyp:smoothness} and~\ref{hyp:elliptic}.
We used H\"older's inequality in the first line, a Cauchy--Schwarz inequality with the same reweighting argument as before in the second line, and~$t-s\leq T$ in the last line.

Collecting the estimates~\eqref{eq:i1_estimate}~\eqref{eq:i2_estimate},~\eqref{eq:I3_martingale_term},~\eqref{eq:I3_boundary_terms_final_bound} and~\eqref{eq:I3_bv_integral_term}, we find that there exists~$\gamma>2$ and~$C>0$ such that, for all~$\lambda\geq 1$ and all~$0\leq s < t\leq T$, it holds
\begin{equation}
    \label{eq:tightness_final_estimate}
    \E_{\mu_0}\left[\left|X_t^\lambda-X_s^\lambda\right|^\gamma\right] \leq C(t-s)^{\gamma/2}.
\end{equation}
By the criterion~\eqref{eq:tightness_criterion}, the family of path laws~$\left(X^\lambda_*\P_{\mu_0}\right)_{\lambda\geq 1}$ is tight, which concludes the proof of~\eqref{eq:convergence_in_law} by a corollary of Prokhorov's theorem, see~\cite[Theorem 8.1]{B68}.
\paragraph{Proof of Corollary~\ref{corr:traj_averages}.}
We simply write
\begin{equation}
    \begin{aligned}
    \E_{\mu_0}\left[\left|\frac1T\int_0^T \varphi(X_s^\lambda)\,\d s - \frac1T\int_0^T \varphi(X_s)\,\d s\right|^r\right] & \leq \frac1T\int_0^T\E_{\mu_0}\left[\left|\varphi(X_s^\lambda)-\varphi(X_s)\right|^r\right]\,\d s\\
    &\leq \|\varphi\|_{\cC^{0,\eta}(\cX)}^{r/\eta}\underset{0\leq s\leq T}{\sup}\,\E_{\mu_0}\left[\left|X_s^\lambda-X_s\right|^{r/\eta}\right]
    \end{aligned}
\end{equation}
and apply~\eqref{eq:convergence_in_l2} with~$\alpha = r/\eta$.

\section{Overdamped limit of coarse-grained Langevin dynamics}
\label{sec:coarse_graining}
    It is possible, by mimicking the arguments of Section~\ref{sec:proof_main}, to the study the case where the underdamped dynamics is given by
    \begin{equation}
        \label{eq:effective_ul}
        \left\{\begin{aligned}
        \d z_t^\lambda &= A(z_t^\lambda)\,v_t^\lambda\,\d t,\\
        \d v_t^\lambda &= \left(-A\nabla V + \frac1\beta\,\div A\right)(z_t^\lambda)\,\d t - \lambda v_t^\lambda\,\d t + \sqrt{\frac{2\lambda}{\beta}}\,\d W_t^\lambda
        \end{aligned}\right.
    \end{equation}
    where~$A:\cX\to\cSpd$ is a positive-definite gradient field. Such equations appear naturally when considering effective (or ``coarse-grained'') kinetic Langevin dynamics, in the spirit of the works~\cite{LL10,ZHS16} for the overdamped case, as explained in the upcoming work~\cite{LSZ26}.

    The process~$z^\lambda$ can be understood as a dynamical approximations of the position process~$q^\lambda$ of the kinetic Langevin dynamics~\eqref{eq:underdamped_friction_dependent} (with~$M=D=\Id$) on~$\R^p$ through a collective variable~$\xi:\R^p\to\cX$ (typically with~$p\gg \dim\,\cX$).
    In this perspective,~$V$ is the free energy associated with $\xi$ , and~$A(z)$ is given by the conditional average
    $$
    A(z) = \int_{\xi^{-1}(z)} G^{1/2}\,\d\nu_z
    $$
    where~$\nu_z$ is the conditional Gibbs measure on the level set~$\xi^{-1}(z)$ and~$G = \nabla\xi^\top\nabla\xi$ is the Gram matrix (we refer to~\cite[Section 3.2.1]{LRS10} for further details on the physical meaning of these objects).

    The Boltzmann--Gibbs measure~\eqref{eq:boltzmann_gibbs} (with kinetic energy~$U(v)=\frac12|v|^2$) is the unique invariant probability measure for the dynamics~\eqref{eq:effective_ul}, whose generator is given by
    \begin{equation}
        \label{eq:cg_generator}
        \widetilde{\cL}_\lambda = \frac1\beta\left(\nabla_v^*A\nabla_z-\nabla_z^*A\nabla_v-\lambda\nabla_v^*\nabla_v\right).
    \end{equation}

    The proofs of Theorem~\ref{thm:overdamped_limit} and Corollary~\ref{corr:pathwise} given above can be straightforwardly adapted to obtain the overdamped limit of the dynamics~\eqref{eq:effective_ul}, by the following result.
    \begin{theorem}
        \label{thm:effective_od}

     Let~$Z$ be the solution to
    \begin{equation}
        \label{eq:effective_od}
        \d Z_t = -A^2(Z_t)\nabla V(Z_t)\,\d t +\frac1\beta\div\left(A^2\right)(Z_t)\,\d t +\sqrt{\frac{2}{\beta}}A(Z_t)\,\d W_t,
    \end{equation}
      which is none other than the dynamics~\eqref{eq:overdamped_friction_dependent} with diffusion matrix~$D=A^2$.

    Assume that the conditions of Theorem~\ref{thm:overdamped_limit} are satisfied, and that~$Z_0^\lambda = Z_0$ for all~$\lambda>0$ with~$Z_0\sim \mu_0$.
    Assume furthermore that Assumption~\ref{hyp:elliptic} is satisfied for the choice~$D=A$.

    Let~$T>0$.
    \begin{enumerate}[i)]
            \item{
            For any~$0<\alpha\leq 2/q$, in the limit~$\lambda\to +\infty$,
            \begin{equation}
                \label{eq:effective_wasserstein_bound}
        \underset{0\leq t\leq T}{\sup}\,\E_{\mu_0}\left[\left|Z_t^\lambda - Z_t\right|^\alpha\right] = \bigo\left(\lambda^{-\alpha/2}\right).
    \end{equation}}
            \item{
            Pathwise weak convergence holds:
                \begin{equation}
            \label{eq:effective_pathwise_cv}
        Z_*^\lambda\P_{\mu_0} \xrightarrow[\lambda\to +\infty]{\text{weakly}} Z_*\P_{\mu_0}\,\qquad \text{in }\cP\left(\cC\left([0,T];\cX\right)\right).
    \end{equation} }
    \end{enumerate}
    \end{theorem}

    We sketch the proof of this result, omitting some details for the sake of brevity, since the arguments are minor modifications of those contained in the proofs of Theorem~\ref{thm:effective_od} and Corollary~\ref{corr:pathwise}.

    \begin{proof}[(Sketch of proof)]
    We first note that since~$\nabla^2 U = \Id$, the generator~$\widetilde{\cL}_\lambda$ is hypoelliptic for any~$\lambda>0$, by a straightforward adaptation of the argument in the proof of Lemma~\ref{lemma:hypoellipticity}, using the fact that~$A$ has full rank everywhere, which follows from Assumption~\ref{hyp:elliptic}.

    Applying It\^o's formula to the process~$A(z_t^\lambda)v_t^\lambda$, we find
    \begin{equation}
        \d\left(A(z_t^\lambda) v_t^\lambda\right) = \left[-A^2\nabla V + \frac{A}\beta\,\div\, A\right](z_t^\lambda)\,\d t - \lambda\,\d z_t^\lambda + \sqrt{\frac{2\lambda}\beta}A(z_t^\lambda)\,\d W_t^\lambda + A'(z_t^\lambda)\left[A(z_t^\lambda)v_t^\lambda\right]v_t^\lambda\,\d t.
    \end{equation}
    With a direct computation similar to~\eqref{eq:ipp_fonda}, we show
    \begin{equation}
        \Pi_0\left(A'(z)\left[A(z)v\right]v\right) = \frac1\beta\left(\div(A^2)-A\,\div\,A\right)(z),
    \end{equation}
    so that, writing~$Z_t^\lambda = z_{\lambda t}^\lambda$ for the time-rescaled process, the equality
    \begin{equation}
        Z_t^\lambda = Z_0 - \int_0^t \left[A^2(Z_s^\lambda)\nabla V(Z_s^\lambda) - \frac1\beta\,\div A^2(Z_s^\lambda)\right]\,\d s + \sqrt{\frac{2}{\beta}}\int_0^t A(Z_s^\lambda)\,\d W_s + \widetilde{R}(t,\lambda),
    \end{equation}
    holds, with remainder term
    \begin{equation}
        \widetilde{R}(t,\lambda) = \frac{A(z_0^\lambda)v_0^\lambda - A(z_{\lambda t}^\lambda)v_{\lambda t}^\lambda}{\lambda} + \frac1\lambda\int_{0}^{\lambda t}\widetilde{\psi}(z_s^\lambda,v_t^\lambda)\,\d s,
    \end{equation}
    where~$\widetilde{\psi}(z,v) = (\Id-\Pi_0)A'(z)[A(z)v]v$.

    Using the following result, which is the counterpart of Lemma~\ref{lemma:hypocoercivity} in this setting, the Gr\"onwall estimate and reweighting argument from Theorem~\ref{thm:overdamped_limit} concludes the proof of the first item in Theorem~\ref{thm:effective_od}.
    \begin{lemma}
        \label{lemma:hypocoercivity_bis}
        Suppose that Assumption~\ref{hyp:elliptic} with~$D=A$ and Assumption~\ref{hyp:UVreg} are satisfied.
        Then the conclusions of Lemma~\ref{lemma:hypocoercivity} remain valid with the substitution~$\cL_\lambda \leftarrow \widetilde{\cL}_\lambda$.
    \end{lemma}
    For completeness, we prove Lemma~\ref{lemma:hypocoercivity_bis} in Appendix~\ref{sec:proofs} below.

    To obtain the second item in Theorem~\ref{thm:effective_od}, we solve a linear SDE for~$v^\lambda$ and find
    \begin{equation}
        v_t^\lambda = \e^{-\lambda t}v_0 + \int_0^t \e^{-\lambda(t-s)}\left[\left(-A\nabla V + \frac1\beta\div\,A\right)(z_s^\lambda)\,\d s + \sqrt{\frac{2\lambda}{\beta}}\,\d W_s^\lambda\right].
    \end{equation}
    Then, for~$0\leq s < t\leq T$ and~$a>1$,
    \begin{equation}
        \left|Z_t^\lambda - Z_s^\lambda\right|^a \leq \|A\|_{L^\infty_{\mathrm{op}}}^a\left|\int_{\lambda s}^{\lambda t}v_r^\lambda\,\d r\right|^a\leq \|A\|_{L^\infty_{\mathrm{op}}}^a3^{a-1}\left(\left|\widetilde{I}_1(s,t)\right|^a+\left|\widetilde{I}_2(s,t)\right|^a+\left|\widetilde{I}_3(s,t)\right|^a\right).
    \end{equation}
    Controlling each of these three summands poses no additional difficulty compared to the previous analysis in which~$\lambda$ was a function of~$z$.
    In fact, since the fundamental solution process~$\widetilde{\cR}^\lambda_t = \e^{-\lambda t}$ is deterministic here, one can control~$|\widetilde{I}_3(s,t)|^a$ using the more direct Fubini argument of~\cite[Lemma 3.1]{WSW24} (which is correct in this case) instead of the computations following~\eqref{eq:i3_rewrite}.
    The resulting estimates imply the tightness of the family of path measures~$(Z^\lambda_*\P_{\mu_0})_{\lambda>0}$, and therefore the desired convergence of path distributions.

    Similar arguments should extend straightforwardly to underdamped dynamics of the form
        \begin{equation}
        \label{eq:effectiwe_ul_mod}
        \left\{\begin{aligned}
        \d y_t^\lambda &= A(y_t^\lambda)\,\nabla U(w_t^\lambda)\,\d t,\\
        \d w_t^\lambda &= \left(-A\nabla w + \frac1\beta\,\div A\right)(y_t^\lambda)\,\d t - \lambda D^{-1}(y_t^\lambda) \nabla U(w_t^\lambda)\,\d t + \sqrt{\frac{2\lambda}{\beta}}D^{-1/2}(y_t^\lambda)\,\d W_t^\lambda,
        \end{aligned}\right.
    \end{equation}
    \end{proof}
    with the second item in Theorem~\ref{thm:effective_od} generalizing at least to quadratic~$U$.

    The result of Theorem~\ref{thm:effective_od} is important from a modelling perspective, because it shows that the underdamped and coarse-grained approximations are non-commuting in general (see Figure~\ref{fig:non_commutativity}).
    Indeed, if we start from the dynamics~\eqref{eq:underdamped_friction_dependent} with~$D=M=\Id$ and take the overdamped limit~$\lambda\to +\infty$, we obtain the standard overdamped Langevin equation~(\eqref{eq:overdamped_friction_dependent} with~$D=\Id$). If one then takes the coarse-grained effective dynamics through~$\xi$, it is well-known~(see~\cite{LL10,ZHS16}) that the resulting dynamics is again of the form~\eqref{eq:overdamped_friction_dependent}, where the potential~$V$ is given by the free energy, and the diffusion~$D$ is given by the conditional expectation
    \begin{equation}
        a_\xi(z) = \int_{\xi^{-1}(z)}G\,\d\nu_z\neq \left(\int_{\xi^{-1}(z)} G^{1/2}\,\d\nu_z\right)^2 = A^2(z),
    \end{equation}
so that taking the overdamped limit of the coarse-grained underdamped dynamics gives a different result in general. The two approximations may nevertheless be close if the conditional (matrix-valued) variances~$\left({\mathrm{Var}}_{\nu_z}\left(G^{1/2}\right)\right)_{z\in\cX}$ are sufficiently small (for instance, they vanish when~$\xi$ is a linear map).

\begin{figure}
  \centering
  \[
    \begin{tikzcd}[row sep=5.5em, column sep=5.5em]
      (q_t^\lambda,p_t^\lambda)
        \arrow[r, "{\lambda\to +\infty,\; t\leftarrow \lambda t}"]
        \arrow[d, "{\text{CG}(\xi)}"']
      & X_t
        \arrow[d, dashed, red, "{\text{CG}(\xi)}"]
      \\
      (z_t^\lambda,v_t^\lambda)
        \arrow[r, "{\lambda\to +\infty,\; t\leftarrow \lambda t}"]
      & Z_t
    \end{tikzcd}
  \]
  \caption{Non-commutation of the overdamped and effective/coarse-grained dynamical approximations. Horizontal arrows correspond to overdamped approximations, while vertical arrows correspond to taking effective dynamics through the collective variable~$\xi$.
  The coarse-graining of~$X$ through~$\xi$ is not equal to~$Z$ for general~$\xi$.}
  \label{fig:non_commutativity}
\end{figure}
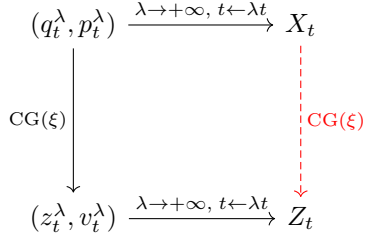

\section{Overdamped limit of some Langevin equations with position-dependent mass matrices}
\label{sec:mass_matrices}
In this section, we apply Theorem~\ref{thm:overdamped_limit} and Corollary~\ref{corr:pathwise} to study a case where the mass itself is position-dependent. In contrast to the dynamics~\eqref{eq:underdamped_friction_dependent} in Section~\ref{sec:intro}, the Hamiltonian is no longer separable into a sum of kinetic and potential terms.
We consider dynamics of the form
\begin{equation}
    \label{eq:general_langevin}
    \left\{\begin{aligned}
    \d q_t^\lambda &= \nabla_p H_M(q_t^\lambda,p_t^\lambda)\,\d t,\\
    \d p_t^\lambda &= -\nabla_q H_M(q_t^\lambda,p_t^\lambda)\,\d t -\lambda\Sigma(q_t^\lambda)\nabla_p H_M(q_t^\lambda,p_t^\lambda)\,\d t + \sqrt{\frac{2\lambda}{\beta}}\Sigma(q_t^\lambda)^{1/2}\,\d W_t^\lambda,
    \end{aligned}\right.
\end{equation}
on~$\cX\times\R^d$, where~$H_M$ is given by the Hamiltonian
\begin{equation}
    \label{eq:modified_hamiltonian}
    H_M(q,p) := \frac12 p^\top M^{-1}(q) p + E(q),\qquad E(q)=V(q) +\frac1{2\beta} \log\det M(q),
\end{equation}
with~$M,\Sigma:\cX\to\cSpd$ are smooth, symmetric positive-definite matrix fields. Physically, the variables~$(q,p)$ correspond respectively to the positions and momenta, and the matrix~$M$ to a position-dependent ``mass'' parameter. Under the corresponding Boltzmann--Gibbs measure
\begin{equation}
    \label{eq:mass_dependent_boltzmann_gibbs}
    \mu_M(q,p)\,\d q\,\d p = \frac1{Z}\e^{-\beta H_M(q,p)}\,\d q\,\d p,\qquad Z = (2\pi/\beta)^{d/2}\int_{\cX} \e^{-\beta V(q)}\,\d q
\end{equation}
the momentum variable~$p$ is distributed according to the Gaussian distribution~$\mathcal N\left(0,\frac{1}{\beta}M(q)\right)$ conditionally on the position variable~$q$.
The normalization constant in~\eqref{eq:mass_dependent_boltzmann_gibbs} is independent of~$M$, thanks to the term~$\frac1{2\beta}\log\det M(q)$ appearing in the potential~$E$.
Since~$H_M$ is non-separable, the Boltzmann--Gibbs measure can no longer be written as a product measure. However, the marginal in~$q$ is again the Gibbs measure~\eqref{eq:gibbs}, i.e.
\begin{equation}
    \int_{\R^d}\mu_M(q,p)\,\d p = \nu(q).
\end{equation}
For this reason, the dynamics~\eqref{eq:general_langevin} can be used to sample canonical configurations.
It is a natural alternative to~\eqref{eq:underdamped_friction_dependent} if one views~$M$ as a preconditioner, adapting the covariance in the momentum variable to anisotropic features of the target distribution.
In the case of a scalar friction~$\Sigma = \gamma>0$, these dynamics are related to the Riemann manifold Hamiltonian Monte Carlo (RMHMC) family of sampling algorithm, see~\cite{GC11} and~\cite[Section 3.2]{LSS24a} for in-depth discussions.
Position-dependent mass matrices also appear naturally when considering internal-coordinate molecular dynamics, see~\cite{VJ15}, in  which the state of the system is parametrized by ``internal'' degrees of freedom, such as nuclear bond lengths and torsion angles. In this setting, the dynamics~\eqref{eq:general_langevin} is the natural counterpart to the Cartesian equation~\eqref{eq:underdamped_friction_dependent}.

\paragraph{A canonical transformation.}
Here, we only consider the class of dynamics which can be transformed into one of the form~\eqref{eq:underdamped_friction_dependent} through a smooth diffeomorphism
\begin{equation}
    \label{eq:change_of_variables}
    \Gamma(q,p) = \left(x,v\right) = \left(x(q,p),v(q,p)\right).
\end{equation}
We restrict ourselves to transformations~$\Gamma$ satisfying the following three conditions.
\begin{itemize}
    \item{To come back to the original coordinates from the overdamped equation in $x$, the transformed position~$x\in\cX$ from~\eqref{eq:change_of_variables} should depend only on~$q$, i.e.
    \begin{equation}
    \label{eq:x_const_wrt_p}
    \nabla_p x = 0.
    \end{equation}}
    \item{In order to preserve the Hamiltonian nature of the dynamics in the case~$\lambda=0$, the change of variables should be a canonical transformation, meaning that it should satisfy locally the symplecticity condition
    \begin{equation}
    \label{eq:symplecticity}
    \nabla \Gamma^\top J \nabla\Gamma = J,\qquad J = \begin{pmatrix}
        0 & \mathrm{Id}_{\R^d}\\ -\mathrm{Id}_{\R^d}&0
    \end{pmatrix},
\end{equation}
see for example~\cite[Chapter 8]{A89}.}
    \item{
        Finally, to apply the pathwise convergence result of Theorem~\ref{thm:overdamped_limit}, the Hamiltonian should become separable, with the kinetic energy given by a quadratic form
        \begin{equation}
            \label{eq:separability}
            H_M(q,p) = \frac12 v^\top G_M v + V_M\circ x
        \end{equation}
        for some constant matrix~$G_M\in\cSpd$ and~$V_M:\cX\to\R$.
    }
\end{itemize}

Writing the symplecticity condition, we find
\begin{equation}
    \begin{pmatrix}
        \nabla_q x & 0 \\ \nabla_q v & \nabla_p v
    \end{pmatrix}^\top J \begin{pmatrix}
        \nabla_q x & 0 \\ \nabla_q v & \nabla_p v
    \end{pmatrix} = \begin{pmatrix}
    \nabla_q x^\top\nabla_q v -\nabla_q v^\top \nabla_q x & \nabla_q x^\top \nabla_p v \\ -\nabla_p v^\top\nabla_q x & 0
    \end{pmatrix},
\end{equation}
which imposes (since~$\nabla_q x$ must be invertible for~$\Gamma$ to be a diffeomorphism) that~${\nabla_q x = \left(\nabla_p v\right)^{-\top}}$, and that the matrix~$\left[\nabla_q x^\top \nabla_q v\right](q,p)$ is symmetric.
The first condition and~\eqref{eq:x_const_wrt_p} impose~${\nabla^2_p v=0}$, so that~${v(q,p)=A(q) p+b(q)}$ is affine in~$p$. Substituting this ansatz into~\eqref{eq:separability}, equating linear and quadratic parts in~$p$, we find
\begin{equation}
    \label{eq:factorization}
   v(q,p)=A(q)p,\qquad\forall\,q\in \cX,\qquad A(q)^{-\top} M^{-1}(q)A(q)^{-1} = G_M,\qquad b(q)=0.
\end{equation}
Then,~$\nabla_q x = A(q)^{-\top}$, so that the rows of~$A(q)^{-\top}$ must be gradients of smooth functions on~$\cX$. Therefore,~$x$ can be defined by
\begin{equation}
    \label{eq:q_definition}
    x(q) = q_0 + \int_{0}^1 A^{-\top}(\gamma(t))\gamma'(t)\,\d t,\qquad\forall\,\gamma\in\left\{f\in\mathcal C^\infty([0,1];\cX):\, f(0)=q_0,f(1)=q\right\},
\end{equation}
for any choice of base point~$q_0=x(q_0)$ and path~$\gamma$. One can check with some index computation (see Lemma~\ref{lemma:commutation}) that this condition in turn implies the local symmetry property
\begin{equation}
    \label{eq:symmetry_condition}
\nabla_q x^\top \nabla_q v = A^{-1}(q)\nabla_q v = \nabla_q v^\top \nabla_q x = \nabla_q v^\top A^{-\top},
\end{equation}
so that the symplecticity condition~\eqref{eq:symplecticity} is indeed verified for this choice of~$(x,v)$.

Various classes of mass matrices ensure that the factorization~\eqref{eq:factorization} holds, with~$A^{-\top}$ having gradient rows. When~$\cX=\R^d$, one can simply choose~$M=(\nabla^2 \Phi)^2$ for some smooth potential~$\Phi:\R^d\to\R$ with non-degenerate Hessian, in which case~$A=(\nabla^2 \Phi)^{-1}$ and~$G_M=\Id$, or~$M = (\Id + \nabla\theta)(\Id+\nabla\theta)^\top$ for some sufficiently small vector field~$\theta$, in which case~$A = (\Id+\nabla\theta)^{-1}$ and again~$G_M=\Id$.
More generally, this construction is valid when~$M$ factors as~$M=\nabla\Psi\nabla\Psi^\top$ for some smooth diffeomorphism~$\Psi:\cX\to\cX$.

\paragraph{An It\^o computation.}
We assume in this paragraph that the transformation~$\Gamma$ defined in~\eqref{eq:change_of_variables} by~\eqref{eq:factorization} by~\eqref{eq:q_definition} is both valid and smooth.

Denote by~$(x_t^\lambda,v_t^\lambda) := \Gamma(q_t^\lambda,p_t^\lambda)$. We apply It\^o's lemma to~\eqref{eq:general_langevin} in order to derive an SDE in these new variables.
The first simplification follows from noting that, since~$\Gamma(q,p)$ is linear with respect to~$p$, the partial Hessian of each component of~$\Gamma$ with respect to~$p$ vanishes -- this will simplify the computation by dropping the quadratic covariation terms.

We first write, using~$\nabla_q x = A^{-\top}(q)$,
\begin{equation}
    \d x_t^\lambda = \nabla_q x(q_t^\lambda)\nabla_p H_M(q_t^\lambda,p_t^\lambda)\,\d t = A^{-\top}(q_t^\lambda)M^{-1}(q_t^\lambda)p_t^\lambda\,\d t = G_M v_t^\lambda\,\d t,
\end{equation}
using~\eqref{eq:factorization} and the form~\eqref{eq:modified_hamiltonian} of the Hamiltonian in the third equality. Next,
\begin{equation}
    \begin{aligned}
    \d v_t^\lambda &= \nabla_q v(q_t^\lambda,p_t^\lambda)\,\d q_t^\lambda + \nabla_p v(q_t^\lambda,p_t^\lambda)\,\d p_t^\lambda\\
    &= \nabla_q v(q_t^\lambda,p_t^\lambda)M^{-1}(q_t^\lambda)p_t^\lambda\,\d t\\
    &+ A(q_t^\lambda)\left(-\nabla_q H_M(q_t^\lambda,p_t^\lambda)\,\d t -\lambda\Sigma(q_t^\lambda)\nabla_p H_M(q_t^\lambda,p_t^\lambda)\,\d t + \sqrt{\frac{2\lambda}{\beta}}\Sigma(q_t^\lambda)^{1/2}\,\d W_t^\lambda\right)\\
    \end{aligned}
\end{equation}
The fluctuation-dissipation term reads
\begin{equation}
    -\lambda A(q_t^\lambda)\Sigma(q_t^\lambda)M^{-1}(q_t^\lambda)p_t^\lambda\,\d t + \sqrt{\frac{2\lambda}{\beta}}A(q_t^\lambda)\Sigma(q_t^\lambda)^{1/2}\,\d W_t^\lambda = -\lambda \widetilde{\Sigma}(q_t^\lambda)G_M v_t^\lambda \,\d t+ \sqrt{\frac{2\lambda}{\beta}}\widetilde{\Sigma}(q_t^\lambda)^{1/2}\,\d W_t^\lambda,
\end{equation}
where~$\widetilde{\Sigma}(q) = A(q)\Sigma(q)A(q)^\top$, and~$\widetilde{\Sigma}(q)^{1/2}:=A(q)\Sigma(q)^{1/2}$ is clearly a square root for~$\widetilde{\Sigma}(q)$.

For the contribution of the force term, we write~$\nabla_q H_M(q,p) = \frac12 \nabla_q\left(p^\top M^{-1}(q)p\right) + \nabla E(q)$. Note that
\begin{equation}
    -A(q_t^\lambda)\nabla E(q_t^\lambda)\,\d t = -\left[\nabla_q x\right]^{-\top}\nabla_q E(q_t^\lambda)\,\d t = -\nabla_x\left(E\circ q\right)(x_t^\lambda)\,\d t,
\end{equation}
using~$A(q)^\top = \left[\nabla_q x\right]^{-1}(q)$, leaving a final term
\begin{equation}
    \nabla_q v(q_t^\lambda,p_t^\lambda)M^{-1}(q_t^\lambda)p_t^\lambda\,\d t - \frac12 A(q_t^\lambda)\nabla_q\left(p_t^{\lambda\top} M^{-1}(q_t^\lambda)p_t^\lambda\right)\,\d t.
\end{equation}
We write
\begin{equation}
    \begin{aligned}
        \frac12 A(q)\nabla_q\left(p^\top M^{-1}(q)p\right) &= \frac12 A(q)\nabla_q\left(v^\top G_M v\right)\\
        &= A(q)\left(\nabla_q v\right)^\top G_M v,
    \end{aligned}
\end{equation}
as well as
\begin{equation}
    \begin{aligned}
        \nabla_q v M^{-1}(q)p &= \nabla_q v A(q)^\top A(q)^{-\top} M^{-1}(q)A^{-1}(q)A(q)p\\
        &= \nabla_q v A(q)^\top G_M v.
    \end{aligned}
\end{equation}
Multiplying the symmetry condition~\eqref{eq:symmetry_condition} on the left by~$A(q)$ and on the right by~$A(q)^\top$, we see that~$\nabla_q v A(q)^\top=A(q)\nabla_q v^\top$, which shows the remaining term is zero.

Collecting terms,~$(x_t^\lambda,v_t^\lambda)$ satisfies the  following Langevin equation
\begin{equation}
    \label{eq:final_equation}
    \left\{\begin{aligned}
    \d x_t^\lambda &= G_M v_t^\lambda\,\d t,\\
    \d v_t^\lambda &= -\nabla V_M(x_t^\lambda)\,\d t - \lambda\Sigma_M(x_t^\lambda) G_M v_t^\lambda\,\d t + \sqrt{\frac{2\lambda}{\beta}}\Sigma_M(x_t^\lambda)^{1/2}\,\d W_t^\lambda,
    \end{aligned}\right.
\end{equation}
which is Equation~\eqref{eq:underdamped_friction_dependent} for the modified potential~$V=V_M$, kinetic energy~$U(v)=\frac12 v^\top G_M v$, and friction matrix~$D=\Sigma_M$ defined by
\begin{equation}
    \label{eq:modified_coefficients}
    V_M(x) = \left[E\circ q\right](x),\qquad \Sigma_M(x) = \left[A\Sigma A^\top\circ q\right](x),\qquad \Sigma_M(x)^{1/2} = \left[A\Sigma^{1/2}\circ q\right](x).
\end{equation}

\paragraph{Overdamped limit.}
We apply Corollary~\ref{corr:pathwise} to the transformed equation~\eqref{eq:final_equation}.
\begin{corollary}
    \label{thm:overdamped_mass_matrix}
    Assume that~$V:\cX\to\R$ satisfies Assumptions~\ref{hyp:smoothness} and~\ref{hyp:UVreg}, and that the transformation~$\Gamma$ defined in~\eqref{eq:change_of_variables} is a smooth symplectic transformation, where~$M$ satisfies a factorization of the form~\eqref{eq:factorization}.
    
    Assume moreover that~$M$,~$A$,~$\Sigma$ are smooth and~$\cW^{2,\infty}$, and that Assumption~\ref{hyp:elliptic} is satisfied for~$D=M$ (and therefore also for~$D=AA^\top$) and~$D=\Sigma^{-1}$.
    
    Suppose that, for each~$\lambda>0$, $(q_0^\lambda,p_0^\lambda)=(q_0,p_0)$, with~$(q_0,p_0)\sim \mu_{0}\in\cP(\cX\times\R^d)$, such that~$\mu_{0}\ll \mu_M$, and~$\d\mu_{0}/\d \mu_M\in L^p(\mu_{M})$ for some~$p>1$.

    For any~$\lambda>0$, let~$(q_t^\lambda,p_t^\lambda)_{t\geq 0}$ be the solution to~\eqref{eq:general_langevin}. Then, denoting~$(Q_t^\lambda)_{t\geq 0}=(q_{\lambda t}^\lambda)_{t\geq 0}$, it holds
    \begin{equation}
        Q^\lambda_*\P_{\mu_0} \xrightarrow[\lambda\to +\infty]{\text{weakly}} Q^M_*\P_{\mu_0}\,\text{ in }\cP(\mathcal C([0,T],\cX)),
    \end{equation}
    where~$Q^M_t = q(Z^M_t)$ and~$Z^M$ solves the SDE
    \begin{equation}
        \label{eq:od_mass}
        \d Z_t^M = -\Sigma_M^{-1}(Z_t^M)\nabla V_M(Z_t^M)\,\d t + \frac1\beta \,\mathrm{div}\,\Sigma_M^{-1}(Z_t^M)\,\d t + \sqrt{\frac{2}\beta}\Sigma_{M}^{-1/2}(Z_t^M)\,\d W_t,
    \end{equation}
    with~$Z_0^M\sim\nu$ independent from~$W$.
\end{corollary}
\begin{proof}
Denote~$Z^{M,\lambda}_t = x_{\lambda t}^\lambda$, where~$(x^\lambda,v^\lambda)$ solves the SDE~\eqref{eq:final_equation}.

Since~$M$ and~$\log\det M$ are uniformly bounded on~$\cX$, it is straightforward to check that the coefficients~$V_M$ and $\Sigma_M$ in Equation~\eqref{eq:final_equation} and the kinetic energy~$U_M(v)=\frac12v^\top G_M v$ satisfy Assumptions~\ref{hyp:smoothness},~\ref{hyp:elliptic} and~\ref{hyp:UVreg} of Theorem~\ref{thm:overdamped_limit} (in particular, the Poincaré inequality in~$\nu$~\eqref{eq:poincare} is stable under bounded perturbations of the potential energy~$V$).
Furthermore, Assumption~\eqref{hyp:phys_kinetic} is vacuously satisfied, as well as Assumption~\ref{hyp:hypoellipticity} using Lemma~\ref{lemma:hypoellipticity}.

Since~$\Gamma$ is a symplectic transformation,~$|\det\,\nabla \Gamma|=1$, and one easily checks that~$\Gamma_*\mu_0 \ll \mu$ with~${\frac{\d(\Gamma_*\mu_0)}{\d\mu} = \left(\frac{\d\mu_0}{\d\mu_M}\right)\circ \Gamma^{-1}\in L^p(\mu)}$, using the Lebesgue change of variables formula, since~$\Gamma_*\mu_M=\mu$. Therefore the condition from Theorem~\ref{thm:overdamped_limit} on the initial distribution~$(x_0^\lambda,v_0^\lambda)\sim \Gamma_*\mu_0$ is satisfied.

Since~$q\circ Z^{M,\lambda}=Q^\lambda$ by the computation of the previous paragraph,~$(q\circ Z^{M,\lambda})_*\P_{\mu_0} = Q^\lambda_*\P_{\mu_0}$ for all~$\lambda\geq 0$. Using the continuity of the mapping~$f\mapsto q\circ f$ on~$\mathcal C([0,T];\cX)$, we take the weak limit of this identity as~$\lambda\to +\infty$ using Corollary~\ref{corr:pathwise}, giving the convergence
$Q^\lambda_*\P_{\mu_0}\to (q\circ Z^M)_*\P_{\mu_0}$ weakly in~$\cP(\mathcal C([0,T];\cX))$, where~$Z^M$ satisfies the SDE~\eqref{eq:od_mass}.
\end{proof}

\paragraph{Computation in the case~$\cX=\R$.}
The limiting dynamics from Corollary~\ref{thm:overdamped_mass_matrix} is not fully explicit, but since~$\Gamma$ is one-to-one, it is closed with respect to the process~$Q_t^M$. 
We find these explicit dynamics in the one-dimensional case~$\cX=\R$, noting that similar computations could be performed in the higher-dimensional case.

We consider the Hamiltonian~\eqref{eq:modified_hamiltonian}, which we denote here by
\begin{equation}
    \frac{p^2}{2m(q)} + w(q) + \frac1{2\beta}\log m(q),
\end{equation}
and the corresponding dynamics~\eqref{eq:general_langevin} given by
\begin{equation}
    \left\{\begin{aligned}
    \d q_t^\lambda &= \frac{p_t^\lambda}{m(q_t^\lambda)}\,\d t,\\
    \d p_t^\lambda &= \frac{(p_t^\lambda)^2}{2}\frac{m'(q_t^\lambda)}{m(q_t^\lambda)^2}\,\d t-w'(q_t^\lambda)\,\d t - \frac1{2\beta}\frac{m'(q_t^\lambda)}{m(q_t^\lambda)}\,\d t -\lambda \frac{\sigma(q_t^\lambda)}{m(q_t^\lambda)}p_t^\lambda\,\d t + \sqrt{\frac{2\lambda\sigma(q_t^\lambda)}{\beta}}\,\d B_t,
    \end{aligned}\right.
\end{equation}
where~$m,m',\sigma,w'$ are smooth and globally Lipschitz on~$\R$,~$\varepsilon\leq m,\sigma\leq \varepsilon^{-1}$ for some~$\varepsilon>0$, and~$B$ is a Wiener process.
The condition~\eqref{eq:factorization} is verified vacuously for~$a(q)=m(q)^{-1/2}$, with
\begin{equation}
    \label{eq:change_of_variables_1d}
    \Gamma(q,p) = (x,v) = \left(\int_0^q\sqrt{m(t)}\,\d t,p/\sqrt{m(q)}\right).
\end{equation}

By Corollary~\ref{thm:overdamped_mass_matrix}, the trajectories of the process~$x \circ q^\lambda$ converge in law in the limit~$\lambda\to +\infty$ to solutions of the one-dimensional SDE
\begin{equation}
    \begin{aligned}
        \label{eq:od_1D_increment}
        \d Z_t &= -q'(Z_t)\frac{w'}{a^2\sigma} (q (Z_t))\,\d t - q'(Z_t)\frac{m'}{2\beta m a^2\sigma}(q(Z_t))\,\d t + \frac{1}{\beta}\left(\frac{1}{a^2\sigma}\circ q\right)'(Z_t)\,\d t + \sqrt{\frac{2}{\beta}}\frac{1}{a\sqrt{\sigma}}(q(Z_t))\,\d B_t\\
        & = -\frac{w'\sqrt{m}}{\sigma} (q (Z_t))\,\d t - \frac{m'}{2\beta\sigma\sqrt{m}}(q(Z_t))\,\d t + \frac{1}{\beta}\left(\frac{m}{\sigma}\circ q\right)'(Z_t)\,\d t + \sqrt{\frac{2 m}{\beta \sigma}}(q(Z_t))\,\d B_t\\
        &=-\frac{w'\sqrt{m}}{\sigma} (q (Z_t))\,\d t - \frac{m'}{2\beta\sigma\sqrt{m}}(q(Z_t))\,\d t + \frac{1}{\beta}\left[\frac{m'}{\sqrt{m}\sigma}-\frac{\sigma'\sqrt{m}}{\sigma^2}\right](q(Z_t))\,\d t +  \sqrt{\frac{2 m}{\beta \sigma}}(q(Z_t))\,\d B_t,
    \end{aligned}
\end{equation}
using~$q' = \frac1{\sqrt{m}}\circ q$ in the last two lines. We now apply It\^o's formula with~$f(z)=q(z)$.
Denoting~$Q_t = q(Z_t)$,~$\d Q_t = q'(Z_t)\d Z_t + \frac{m}{\sigma\beta}(q(Z_t)) q''(Z_t)\,\d t$, substituting~\eqref{eq:od_1D_increment}, and using~${q'' = -\frac{m'}{2m^2}\circ q}$, we find after a simple computation that
\begin{equation}
    \d Q_t = - \left[\frac{w'}{\sigma} +\frac1\beta\frac{\sigma'}{\sigma^2}\right](Q_t)\,\d t + \sqrt{\frac{2}{\beta \sigma}}(Q_t)\,\d B_t,
\end{equation}
which corresponds to the dynamics~\eqref{eq:overdamped_friction_dependent} with the choice~$D=\sigma^{-1}$, and is interestingly independent from the choice of~$m$.

\appendix
\section{Proofs of technical results}
\label{sec:proofs}
In this appendix, we gather the proofs of some useful technical lemmas.
\paragraph{Proof of Lemma~\ref{lemma:hypocoercivity}.}
We prove the uniform-in-$\lambda$ $\wLmuz$-hypocoercivity estimates used in the proof of Theorem~\ref{thm:overdamped_limit}.
\begin{proof}
    We adapt the proof of~\cite[Theorem 3.3]{BFLS22}, and refer to it for further details. Instead of the Ornstein--Uhlenbeck generator~$\mathcal L_{\mathrm{FD}} = -\frac1\beta\nabla_p^*\nabla_p$, the fluctuation-dissipation operator is given by~$\cLou = -\frac1\beta\nabla_p^* D^{-1}\nabla_p$ , and the uniform bound~\eqref{eq:elliptic} will allow to adapt estimates involving~$\cL_{\mathrm{FD}}$ to~$\cLou$.
    Note that since~$\cLham^* \1 = \cLou^* \1 = 0$, both~$\cLham$ and~$\cLou$ preserve~$\wLmuz$.
    As in~\cite{BFLS22}, we write~$\Pi_{\subplus}=\Id-\Pi_0$, and~$A_{\alpha\alpha'}=\Pi_\alpha A\Pi_{\alpha'}$ for any operator~$A$ on~$\wLmuz$ and~$\alpha,\alpha'\in\{0,\subplus\}$. Note that in~\cite[Theorem 3.3]{BFLS22}, the notation $\cS$ corresponds to the symmetric part of the generator, $\gamma\cLou$, while we factor the friction coefficient $\lambda$ out, so that the symmetric part is denoted~$\lambda\cS$ in this work.
    
    To check that~\cite[Assumption 2.2]{BFLS22} holds (which is the microscpic coercivity condition from~\cite{DMS15}), we make the choice
    \begin{equation}
        \label{eq:s_bound}
        s = \lambda\frac{K_\kappa^2}{M_D \beta},
    \end{equation}
    whose validity follows from Assumptions~\ref{hyp:elliptic} and the Poincar\'e inequality~\eqref{eq:poincare} for~$\kappa$.

    The verification of~\cite[Assumption 2.3]{BFLS22} (the macroscopic coercivity condition) is unchanged, since the antisymmetric part~$\cLham$ does not involve~$D$.
    
    Similarly, the microscopic reversibility condition~\cite[Assumption 2.5]{BFLS22} is still satisfied with the momentum-reversal operator $\cR\varphi(q,p)=\varphi(q,-p)$.

    We now check~\cite[Assumption 2.6]{BFLS22}. This involves showing the boundedness of two auxiliary operators, defined by
    \begin{equation}
        A_\lambda=\lambda\Pi_1\cLou\Pi_1,\qquad B_\lambda=\lambda\Pi_2\cLou\cLham\Pi_0\left(\cLham_{\subplus 0}^*\cLham_{\subplus 0}\right)^{-1}+\Pi_2\cLham^2\Pi_0\left(\cLham_{\subplus 0}^*\cLham_{\subplus 0}\right)^{-1},
    \end{equation}
    where~$\Pi_1$ is the orthogonal projector onto~$\mathrm{Ran}\cLham_{\subplus 0} \subset \Pi_{\subplus}\wLmuz$ and~$\Pi_2$ is the projector onto its orthogonal complement in~$\Pi_{\subplus}\wLmuz$.
    The operators $A_\lambda$ and $B_\lambda$ correspond respectively to $\cS_{11}$ and~$\cL_{21}\cA_{10}(\cA_{\subplus 0}^*\cA_{\subplus 0})^{-1}$ in the notation of~\cite{BFLS22}.
    
    To show that~$A_\lambda$ is bounded, we compute (as in~\cite{BFLS22})
    \begin{equation}
        -\cLham_{\subplus 0}^*\cLou\cLham_{\subplus 0} = \frac1\beta \nabla_q^* \mathcal N\nabla_q\Pi_0,\qquad \mathcal N(q) = \int_{\R^d}\nabla^2 U D^{-1}(q)\nabla^2 U\,\d \kappa,
    \end{equation}
    and the entries of this matrix~$\mathcal N$ are bounded as multiplication operators on~$L^2(\nu)$ from Assumption~\ref{hyp:elliptic}. Therefore, the same factorization argument as in~\cite[Theorem 3.3]{BFLS22} shows that~$A_\lambda$ is bounded, with a bound
    \begin{equation}
        \label{eq:Aop_bound}
        \left\|A_\lambda\right\|_{1,1}= \lambda\left\|A_1\right\|_{1,1} < +\infty,
    \end{equation}
    where~$\|\cdot\|_{\alpha,\alpha'}$ denotes the~$\left(\Pi_{\alpha'}\wLmuz\to\Pi_\alpha\wLmuz\right)$-operator norm.

    To show the boundedness of~$B_\lambda=\lambda G+H$, we only need to check that
    \begin{equation}
        \label{eq:Gop}
        G = \Pi_2\cLou\cLham\Pi_0\left(\cLham_{\subplus 0}^*\cLham_{\subplus 0}\right)^{-1}
    \end{equation} is~$\|\cdot\|_{2,0}$-bounded, since
    \begin{equation}
        \label{eq:Hop}
        H = \Pi_2\cLham^2\Pi_0\left(\cLham_{\subplus 0}^*\cLham_{\subplus 0}\right)^{-1}
    \end{equation}
    is a~$\|\cdot\|_{2,0}$-bounded operator which does not involve~$D$, following the same arguments as in the proof of~\cite[Theorem 3.3]{BFLS22}.

    To show the bound on~$G$, we compute
    \begin{equation}
        \lambda\cLou\cLham\Pi_0 = \lambda\mathscr{T}\nabla_q,\qquad \mathscr{T}(q)=-\frac{1}{\beta^2}\nabla_p^*D^{-1}(q)\nabla_p\nabla_p^*\Pi_0,
    \end{equation}
    and the entries of the row-vector~$\mathscr{T}$ are bounded operators on~$\wLmuz$, since~$D^{-1}(q)$ is a matrix of bounded operators on~$L^2(\nu)$, whose entries commute with any of the~$\partial_{p_i}$,~$\partial_{p_i}^*$.
    From here, the boundedness of~$G$ and therefore of~$B_\lambda$ follows using Assumption~\ref{hyp:UVreg}, with the same factorizations and commutation relations as in the proof of~\cite[Theorem 3.3]{BFLS22}.
    In particular we get
    \begin{equation}
        \label{eq:Bop_bound}
        \left\|B_\lambda\right\|_{2,0} \leq \lambda \|G\|_{2,0} + \|H\|_{2,0} < +\infty.
    \end{equation}

    We now show the uniform-in-$\lambda$ estimates~\eqref{eq:hypocoercivity}. In~\cite[Equation 2.3]{BFLS22}, the authors compute the block decomposition of~$\cL_\lambda^{-1}$ on~$\wLmuz=\Pi_0\wLmuz\oplus\Pi_{\subplus}\wLmuz$, which is given by
    \begin{equation}
        \label{eq:block_inverse}
        \cL_\lambda^{-1} = \begin{pNiceArray}{c|c}
        \mathfrak{S}_{0,\lambda}^{-1}&-\mathfrak{S}_{0,\lambda}^{-1}\cLham_{0\subplus}\left[\cL_\lambda^{-1}\right]_{\subplus\subplus}\\
        \hline
        -\left[\cL_\lambda^{-1}\right]_{\subplus\subplus}\cLham_{\subplus 0}\mathfrak{S}_{0,\lambda}^{-1}&\left[\cL_\lambda^{-1}\right]_{\subplus\subplus}+\left[\cL_\lambda^{-1}\right]_{\subplus\subplus}\cLham_{\subplus 0}\mathfrak{S}_{0,\lambda}^{-1}\cLham_{0\subplus}\left[\cL_\lambda^{-1}\right]_{\subplus\subplus}
    \end{pNiceArray},
    \end{equation} 
    where~$\mathfrak{S}_{0,\lambda} = \cLham_{\subplus 0}^*\left[\cL_\lambda^{-1}\right]_{\subplus\subplus}\cLham_{\subplus 0}$ is the Schur complement associated with this decomposition, acting on~$\Pi_0\wLmuz$.
    The proof of~\cite[Theorem 2.7]{BFLS22} shows that all these blocks are bounded operators from~$\wLmuz$ to~$\wLmu$. The top-left block~$\mathfrak{S}_{0,\lambda}^{-1}$ can be shown to be of order~$\lambda$ in operator norm, while the bottom-right block is of order~$\lambda^{-1}$, and the off-diagonal blocks are of order~$1$.
    Therefore, general hypocoercive estimates only provide a bound on the operator norm~$\left\|\cL_\lambda^{-1}\right\|_{\wLmuz\to\wLmuz} =\bigo(\max(\lambda,\lambda^{-1}))$.
    
    However, in the specific case~$\varphi\in \Pi_{\subplus} L_0^2(\mu)$, some terms (namely those corresponding to the left-column of the decomposition~\eqref{eq:block_inverse}) will be dropped from the final estimate for~$\|\cL_\lambda^{-1}\varphi\|_{\wLmu}$, which leads to an estimate~$\left\|\cL_\lambda^{-1}\right\|_{\Pi_{\subplus}\wLmuz\to\wLmuz} = \bigo(1)$ in the limit~$\lambda\to +\infty$.
    More precisely, in this case,
    \begin{equation}
        \cL_\lambda^{-1}\varphi = \left(-\mathfrak{S}_{0,\lambda}^{-1}\cLham_{0\subplus}\left[\cL_\lambda^{-1}\right]_{\subplus\subplus}+\left[\cL_\lambda^{-1}\right]_{\subplus\subplus}+\left[\cL_\lambda^{-1}\right]_{\subplus\subplus}\cLham_{\subplus 0}\mathfrak{S}_{0,\lambda}^{-1}\cLham_{0\subplus}\left[\cL_\lambda^{-1}\right]_{\subplus\subplus}\right)\varphi,
    \end{equation}
    where the proof of~\cite[Theorem 2.7]{BFLS22} shows that
    \begin{equation}
        \label{eq:finalbound_A}
        \left\|\left[\cL_\lambda^{-1}\right]_{\subplus\subplus}+\left[\cL_\lambda^{-1}\right]_{\subplus\subplus}\cLham_{\subplus 0}\mathfrak{S}_{0,\lambda}^{-1}\cLham_{0\subplus}\left[\cL_\lambda^{-1}\right]_{\subplus\subplus}\right\|_{\subplus,\subplus}\leq \frac{2}{s}
    \end{equation}
    in operator norm on~$\wLmuz$, with~$s$ is defined in~\eqref{eq:s_bound}. The remaining term is bounded by
    \begin{equation}
        \label{eq:finalbound_B}
        \left\|\mathfrak{S}_{0,\lambda}^{-1}\cLham_{0\subplus}\left[\cL_\lambda^{-1}\right]_{\subplus\subplus}\right\|\leq \sqrt{\frac{\left\|\mathfrak{S}_{0,\lambda}^{-1}\right\|_{0,0}}{s}}\leq s^{-1/2}\sqrt{\frac{\|A_\lambda\|_{1,1}}{a^2}+b\frac{\|B_\lambda\|_{2,0}}{s}},
    \end{equation}
    where the constants~$a,b$ are independent from~$\lambda$. From the bounds~\eqref{eq:finalbound_A} and~\eqref{eq:finalbound_B}, the first estimate in~\eqref{eq:hypocoercivity} follows from the linear scaling in~$\lambda$ of each of~$s$,~$\|A_\lambda\|$ and~$\|B_\lambda\|$ as~$\lambda\to +\infty$, which follows from~\eqref{eq:s_bound},~\eqref{eq:Aop_bound} and~\eqref{eq:Bop_bound}.

    To estimate~$\nabla_p \cL_\lambda^{-1}\varphi$, we have from an integration by parts that
    \begin{equation}
        \begin{aligned}
            \left\langle\varphi,\cL_\lambda^{-1}\varphi \right\rangle_{\wLmu} &= \lambda\left\langle\cLou\cL_\lambda^{-1}\varphi,\cL_\lambda^{-1}\varphi\right\rangle_{\wLmu}\\
            &= -\frac{\lambda}{\beta}\mu\left(\left[\nabla_p\cL_\lambda^{-1}\varphi\right]^\top D^{-1}\left[\nabla_p\cL_\lambda^{-1}\varphi\right]\right)\\
            &\leq -\frac{\lambda}{\beta M_D}\|\nabla_p \cL_\lambda^{-1}\varphi\|^2_{\wLmu},
        \end{aligned}
    \end{equation}
    using Assumption~\ref{hyp:elliptic} in the last line. From this bound, a Cauchy--Schwarz inequality gives
    \begin{equation}
        \label{eq:final_bound_grad_p}
        \left\|\nabla_p \cL_\lambda^{-1}\varphi\right\|^2_{\wLmu}\leq \frac{\beta M_D}{\lambda}\|\varphi\|_{\wLmu}\|\cL_\lambda^{-1}\varphi\|_{\wLmu} \leq \frac{C_1\beta M_D}{\lambda}\|\varphi\|_{\wLmu}^2,
    \end{equation}
    using the previously derived uniform bound~$\|\cL_\lambda^{-1}\varphi\|_{\wLmu} \leq C_1\|\varphi\|_{\wLmu}$.
\end{proof}

\paragraph{Proof of Lemma~\ref{lemma:hypocoercivity_bis}.}
We prove the uniform-in-$\lambda$ $\wLmuz$-hypocoercivity estimates used in the proof of Theorem~\ref{thm:effective_od}.

\begin{proof}[Proof of Lemma~\ref{lemma:hypocoercivity_bis}]
    In accordance with the notation of Section~\ref{sec:coarse_graining}, we use the notation~$(z,v)$ for the position and momentum variables, instead of~$(q,p)$.

    We again adapt the proof of~\cite[Theorem 3.3]{BFLS22}. The generator can now be decomposed into its antisymmetric and symmetric components as
    \begin{equation}
        \widetilde{\cL}_\lambda = \widetilde{\cA} + \lambda \widetilde{\cS},\qquad \widetilde{\cA} = \frac1\beta\left(\nabla_v^*A\nabla_z - \nabla_z^*A\nabla_v\right),\qquad \widetilde{\cS} = -\frac1\beta \nabla_v^*\nabla v,
    \end{equation}
    and the modification with respect to~\cite{BFLS22} now concerns the antisymmetric part only. Again we factor out the friction coefficient, so that our operator $\lambda\widetilde{\cL}$ corresponds to the operator $\cS$ in the notation of~\cite{BFLS22}.
    
    The verification of~\cite[Assumption 2.1]{BFLS22} is immediate, and that of~\cite[Assumption 2.2]{BFLS22} is unchanged.
    
    The microscopic coercivity constant from~\cite[Assumption 2.2]{BFLS22} can be taken equal to
    \begin{equation}
        \label{eq:micro_coercivity_constant}
        s = \lambda\frac{K_\kappa^2}{\beta}.
    \end{equation}

    To check the macroscopic coercivity condition (\cite{BFLS22}[Assumption 2.3]), we write by a direct computation,
    \begin{equation}
        \widetilde{\cA}^*_{\subplus 0}\widetilde{\cA}_{\subplus 0} = \frac1{\beta^2}\Pi_0\nabla_z^* A\nabla_v\nabla_v^*A\nabla_z\Pi_0
    \end{equation}
    since, for any~$1\leq i,j\leq d$,~$[\left(\nabla_z^* A\right)_i,\partial_{v_j}]=[\left(A\nabla_z\right)_i,\partial_{v_j}^*]=0$, and using~$[\partial_{v_i},\partial_{v_i}^*] = \beta \delta_{ij}$ (recall~$U$ is the standard quadratic energy here), we find
    \begin{equation}
        \label{eq:hyp_2.3_bis}
        \widetilde{\cA}^*_{\subplus 0}\widetilde{\cA}_{\subplus 0} = \frac1\beta\Pi_0\nabla_z^* A^2 \nabla_z\Pi_0
    \end{equation}
    whence, for~$\varphi\in\testfuncs(\cX\times\R^d)$,
    \begin{equation}
        \|\widetilde{\cA}_{\subplus 0}\varphi\|^2 = \frac1\beta\langle A^2\nabla_z\Pi_0\varphi,\nabla_z\Pi_0\varphi\rangle \geq \frac{K_\nu^2}{\beta M_A^2}\|\Pi_0\varphi\|^2,
    \end{equation}
    where we use the Poincar\'e inequality for~$\nu$~\eqref{eq:poincare} and the bound from Assumption~\ref{hyp:elliptic} (with constant~$M_A>0$) in the last inequality.

    The standard choice of the momentum-reversal involution~$\cR:\varphi(q,p)\mapsto \varphi(q,-p)$ again ensures the validity of~\cite[Assumption 2.5]{BFLS22}, so it remains to show~\cite[Assumption 2.6]{BFLS22}, that is the boundedness of the operators
    \begin{equation}
        \label{eq:def_CD}
        C_\lambda=\lambda\widetilde{\Pi}_1\widetilde{\cS}\widetilde{\Pi}_1,\qquad D_\lambda=\lambda\widetilde{\Pi}_2\widetilde{\cS}\widetilde{\cA}\Pi_0\left(\widetilde{\cA}_{\subplus 0}^*\widetilde{\cA}_{\subplus 0}\right)^{-1}+\widetilde{\Pi}_2{\widetilde{\cA}^2}\Pi_0\left(\widetilde{\cA}_{\subplus 0}^*\widetilde{\cA}_{\subplus 0}\right)^{-1},
    \end{equation}where again~$\widetilde{\Pi}_1,\widetilde{\Pi}_2$ are orthogonal projectors, onto~$\mathrm{Ran}\,\widetilde{\cA}_{\subplus 0}$ and its~$\Pi_{\subplus}\wLmuz$-orthogonal respectively. Some of the computations are in fact simpler than in~\cite{BFLS22}, due to the choice of kinetic energy~$U(v)=\frac12|v|^2$. The operators~$C_\lambda$ and~$D_\lambda$ correspond respectively to~$cS_{11}$ and~$\cL_{21}\cA_{10}(\cA_{\subplus 0}^*\cA_{\subplus 0})^{-1}$ in the notation of~\cite{BFLS22}.

    Using the same commutator identities as the ones leading to~\eqref{eq:hyp_2.3_bis}, we write
    \begin{equation}
        \label{eq:AstarA_identity}
        -\widetilde{\cA}^*_{\subplus 0}\widetilde{S}\widetilde{\cA}_{\subplus 0} = \frac1\beta \Pi_0\nabla_z^*A^2\nabla_z\Pi_0 = \frac1\beta \widetilde{\cA}^*_{\subplus 0}\widetilde{\cA}_{\subplus 0},
    \end{equation}
    and, using the explicit expression
    \begin{equation}
        \widetilde{\Pi}_1 = \widetilde{\cA}_{\subplus 0}\left(\widetilde{\cA}_{\subplus 0}^*\widetilde{\cA}_{\subplus 0}\right)^{-1}\widetilde{\cA}_{\subplus 0}^*,
    \end{equation}
    we find
    \begin{equation}
        \widetilde{\Pi}_1\widetilde{\cS}\widetilde{\Pi}_1 = -\frac1\beta\widetilde{\Pi}_1,\text{ so that }\|C_\lambda\|_{\widetilde{\Pi_1}\wLmuz\to\widetilde{\Pi_1}\wLmuz} \leq \frac{\lambda}{\beta}
    \end{equation}
    in the~$\wLmuz$ operator norm.

    Turning to~$D_\lambda=\lambda J + K_A$, we first bound
    \begin{equation}
        \label{eq:J_def}
        \begin{aligned}
        J=\widetilde{\Pi}_2\widetilde{\cS}\widetilde{\cA}\Pi_0\left(\widetilde{\cA}_{\subplus 0}^*\widetilde{\cA}_{\subplus 0}\right)^{-1}&=-\widetilde{\Pi}_2\nabla_v^*\nabla_v\nabla_v^* A\nabla_z\Pi_0\left({\widetilde{\mathcal A}}^*_{\subplus 0}{\widetilde{\mathcal A}}_{\subplus 0}\right)^{-1}\\
        &=-\Pi_2\nabla_v^*\nabla_v\nabla_v^*\Pi_0A\nabla_z\left({\widetilde{\mathcal A}}^*_{\subplus 0}{\widetilde{\mathcal A}}_{\subplus 0}\right)^{-1}.
        \end{aligned}
    \end{equation}

    As in the proof of~\cite[Theorem 3.3]{BFLS22}, we note that~$\nabla_z\left({\widetilde{\mathcal A}}^*_{\subplus 0}{\widetilde{\mathcal A}}_{\subplus 0}\right)^{-1}$ is component-wise bounded from~$\Pi_0\wLmuz \cong L_0^2(\nu)$ to~$\wLmuz$. This follows from the~$H^1(\nu)$-coercivity property (ensured by Assumption~\ref{hyp:elliptic}) of the bilinear form
    $$a(u,v)=\int_{\cX}\nabla_z u\cdot A^2 \nabla_z v\,\d\nu$$
    associated with~$L=\widetilde{\cA}^*_{\subplus 0}\widetilde{\cA}_{\subplus 0}$, using standard arguments (bound~$a(u,u)$ with~$Lu=f$ from below using coercivity and above using Cauchy--Schwarz and the Poincar\'e inequality for~$\nu$).
    Since~$A$ has uniformly bounded entries, the column of operators~$A\nabla_z\left({\widetilde{\mathcal A}}^*_{\subplus 0}{\widetilde{\mathcal A}}_{\subplus 0}\right)^{-1} $ is also component-wise bounded. The arguments of~\cite[Theorem 3.3]{BFLS22} can then be used verbatim to show that, component-wise, the row of operators~$\nabla_v^*\nabla_v\nabla_v^*\Pi_0$ is bounded, which finally shows that~$J$ is bounded as a product of bounded operators.

    It remains to show that the operator
    $$
    K_A=\widetilde{\Pi}_2{\widetilde{\cA}^2}\Pi_0\left(\widetilde{\cA}_{\subplus 0}^*\widetilde{\cA}_{\subplus 0}\right)^{-1}=\widetilde{\Pi}_2\Pi_{\subplus}{\widetilde{\cA}^2}\Pi_0\left(\widetilde{\cA}_{\subplus 0}^*\widetilde{\cA}_{\subplus 0}\right)^{-1}
    $$
    is~$\Pi_0\wLmuz\to \widetilde{\Pi}_2\wLmuz$-bounded, and in fact we will show the (stronger)~$\|\cdot\|_{\subplus,0}$-boundedness of~$\Pi_{\subplus}{\widetilde{\cA}^2}\Pi_0\left(\widetilde{\cA}_{\subplus 0}^*\widetilde{\cA}_{\subplus 0}\right)^{-1}$.
    We first compute, with~$\varphi\in\testfuncs(\cX\times\R^d)$,
    \begin{equation}
        \widetilde{\cA}^2\Pi_0\varphi = \frac1\beta\left[\left(-\nabla v+\beta v\right)^\top A \nabla_v\left(v^\top A\nabla_z\Pi_0\varphi\right)\right],
    \end{equation}
    whence, using~$\mathrm{Cov}_{\kappa}\left(v\right)=\beta^{-1}\Id$, we find
    \begin{equation}
        \Pi_{\subplus}\widetilde{\cA}^2\Pi_0 = \frac1\beta\left[\mathcal M(A)\cdot\nabla_z + A^2 : \nabla_z^2\right]\Pi_0,\qquad \left(\mathcal M(A)\right)_{i} = \sum_{k,\ell}A_{k\ell}\partial_{\ell}A_{ki}.
    \end{equation}
    It follows that
    \begin{equation}
        \begin{aligned}
    \left\|\Pi_{\subplus}\widetilde{A}^2\Pi_0\left(\widetilde{\cA}_{\subplus 0}^*\widetilde{\cA}_{\subplus 0}\right)^{-1}\Pi_0\varphi\right\|^2_{\wLmuz}&\leq \frac{2}{\beta^2}\|\mathcal{M}(A)\|^2_{L^\infty(\cX;\R^d)}\left\|\nabla_z\Pi_0\left(\widetilde{\cA}_{\subplus 0}^*\widetilde{\cA}_{\subplus 0}\right)^{-1}\Pi_0\varphi\right\|^2_{L^2(\nu)}\\
    &+\frac{2}{\beta^2}\left\|A^2\right\|^2_{L^\infty_{\mathrm{HS}}}\left\|\nabla_z^2\Pi_0\left(\widetilde{\cA}_{\subplus 0}^*\widetilde{\cA}_{\subplus 0}\right)^{-1}\Pi_0\varphi\right\|^2_{L_{\mathrm{HS}}^2(\nu)}.
        \end{aligned}
    \end{equation}
    The first term is controlled using the~$H^1(\nu)$-coercivity argument following~\eqref{eq:J_def}. For the second one, we adapt the regularity estimate from~\cite[Lemma 3.6]{BFLS22}.
    Let~$g = \left(\widetilde{\cA}_{\subplus 0}^*\widetilde{\cA}_{\subplus 0}\right)^{-1}f$, where~$f\in L_0^2(\nu)\cap\testfuncs(\cX)$. It is sufficient to show that, for some~$C>0$,~$\|\nabla^2g\|_{L^2_{\mathrm{HS}}(\nu)}\leq C\|f\|_{L^2(\nu)}$, the general bound following by a density argument.
\end{proof}
From the commutation relation~$[\nabla,\nabla^*] = \beta\nabla^2 V$ (where~$\nabla,\nabla^*$ act on vector fields), we get the formula, for a smooth vector field~$\phi$:
\begin{equation}
    \left\|\nabla^* \phi\right\|^2_{L^2(\nu)} = \sum_{1\leq i,j\leq d}\int_{\cX}(\partial_i \phi_j)(\partial_j\phi_i)\,\d \nu + \beta\int_{\cX}\partial_{ij}^2 V \phi_i\phi_j\,\d \nu = \int_{\cX}\mathrm{Tr}\left[\left(\nabla\phi\right)^2\right]\,\d\nu + \beta\int_{\cX}\phi^\top \nabla^2V\phi\,\d \nu,
\end{equation}
which we apply to~$\phi = A^2\nabla g$ to get the Bochner-like identity:
\begin{equation}
    \label{eq:bochner_identity}
    \beta^2\left\|f\right\|^2_{L^2(\nu)}=\left\|\nabla^*\phi\right\|^2_{L^2(\nu)} = \int_{\cX}\mathrm{Tr}\left[\left(\nabla A^2\nabla g\right)^2\right]\,\d \nu + \beta\int_{\cX}\nabla g^\top A^2\nabla^2 V A^2 \nabla g\,\d\nu,
\end{equation}
where we use~\eqref{eq:AstarA_identity} for the first equality.

We then write
\begin{equation}
    \mathrm{Tr}\left[\left(\nabla A^2\nabla g\right)^2\right]=\mathrm{Tr}\left[\left( A^2\nabla^2 g + D_A g\right)^2\right]
\end{equation}
for some first-order, matrix-valued differential operator~$D_A$ with~$L^\infty(\cX)$ coefficients, since~${A\in\cW^{1,\infty}}$ by Assumption~\ref{hyp:elliptic}. Expanding the square and cycling the trace, we find
\begin{equation}
    \begin{aligned}
        \label{eq:desired_bound}
    \beta^2\left\|f\right\|^2_{L^2(\nu)} &= \int_{\cX}\mathrm{Tr}\left[(A\nabla^2 g A)^2\right]\,\d \nu + T_3(A,g)\\
    &=\left\|A\nabla^2 g A\right\|^2_{L^2_{\mathrm{HS}}(\nu)}+2 \int_\cX \mathrm{Tr}\left[A^2\nabla^2 g D_A g\right]\,\d \nu + \int_{\cX}\mathrm{Tr}\left[\left(D_A g\right)^2\right]\,\d \nu+T_3(A,g)\\
    &\geq \frac1{M_A^4}\left\|\nabla^2 g\right\|^2_{L^2_{\mathrm{HS}}(\nu)}+T_2(A,g)+ T_3(A,g)+T_1(A,g)
    \end{aligned}
\end{equation}
with~$T_2(A,g),T_3(A,g)$ the two rightmost terms in the second line and~$T_1(A,g)$ is the rightmost term in~\eqref{eq:bochner_identity}. Then, for any~$\varepsilon>0$, there exists constants~$C^{(2)}_{\varepsilon,A},C^{(2)'}_{\varepsilon,A}$ independent of~$f$ such that
\begin{equation}
    \begin{aligned}
        \label{eq:t2_bound}
        |T_2(A,g)| &\leq \left|\int_{\cX}\|A^2\nabla^2 g\|_{\mathrm{HS}}\|D_A g\|_{\mathrm{HS}}\,\d \nu\right|\\
        &\leq 2\left\|A^2\nabla^2 g\right\|_{L^2_{\mathrm{HS}}(\nu)}\left\|D_A g\right\|_{L^2_{\mathrm{HS}}(\nu)}\\
        &\leq \varepsilon\left\|\nabla^2 g\right\|^2_{L^2_{\mathrm{HS}}(\nu)}+C^{(2)}_{\varepsilon,A}\|\nabla g\|^2_{L^2(\nu)}\\
        &\leq \varepsilon\left\|\nabla^2 g\right\|^2_{L^2_{\mathrm{HS}}(\nu)} + C^{(2)'}_{\varepsilon,A}\|f\|^2_{L^2(\nu)}
    \end{aligned}
\end{equation}
using two Cauchy--Schwarz inequalities followed by Young's inequality, as well as Assumption~\ref{hyp:elliptic} to uniformly bound the contributions of~$A$. In the last line, we use the~$L_0^2(\nu)\to L_0^2(\mu)$-boundedness of~$\nabla_z\left(\widetilde{A}_{\subplus 0}^*\widetilde{\cA}_{\subplus 0}\right)^{-1}$, see the discussion following~\eqref{eq:J_def}.

Using the same two Cauchy--Schwarz inequalities and the boundedness of the same operator, we find there exists~$C_A^{(3)},C_A^{(3)'}>0$ independent of~$f$ such that
\begin{equation}
    \label{eq:t3_bound}
    \left|T_3(A,g)\right| \leq C_A^{(3)}\|\nabla g\|^2_{L^2(\nu)} \leq C_A^{(3)'}\|f\|^2_{L^2(\nu)}.
\end{equation}

Finally, we proceed, as in the proof of~\cite[Lemma 3.6]{BFLS22}, to find, for any~$\varepsilon>0$ a constant~$C_{A,\varepsilon}^{(1)}>0$ independent from~$f$ and constants~$C_A^{(1)},C_A^{(1)'},C_{V,d}>0$ independent from~$f$ and~$\varepsilon$, such that
\begin{equation}
    \begin{aligned}
    |T_1(A,g)| &= \beta\left|\int_{\cX} \left(A\nabla g\right)^\top \nabla^2 V \left(A \nabla g\right)\right|\,\d \nu \\
    &\leq \beta \left|\int_{\cX}|A\nabla g|^2\|\nabla^2 V\|_{\mathrm{HS}}\,\d \nu\right|\\
    &\leq \beta \|A\|_{L^\infty_{\mathrm{op}}}^2 \left|\int_{\cX}\|\nabla^2 V\|_{\mathrm{HS}}|\nabla g|^2\,\d \nu\right|\\
    &\leq C_A^{(1)}\|\nabla g\|^2_{L^2(\nu)} + C_A^{(1)'}\left|\int_{\cX}|\nabla V||\nabla g|^2\,\d \nu\right|\\
    &\leq C_{A,\varepsilon}^{(1)}\|\nabla g\|^2_{L^2(\nu)} + \varepsilon\||\nabla V||\nabla g|\|^2_{L^2(\nu)}\\
    &\leq C_{A,\varepsilon}^{(1)}\|\nabla g\|^2_{L^2(\nu)} + \varepsilon C_{V,d}\left(\left\|\nabla g\right\|^2_{L^2(\nu)} + \left\|\nabla^2 g\right\|^2_{L_{\mathrm{HS}}^2(\nu)}\right),
    \end{aligned}
\end{equation}
where we use Equation~\eqref{eq:Vreg} from Assumption~\ref{hyp:UVreg} in the fourth line, Young's inequality in the fifth line, and~\cite[Lemma A.24]{V06} in the last line, together with the inequality $|\nabla |\nabla g|| \leq \|\nabla^2 g\|_{\mathrm{HS}}$ (to be understood in a weak sense).

Finally, we use the boundedness of~$\nabla_z\left(\widetilde{A}_{\subplus 0}^*\widetilde{\cA}_{\subplus 0}\right)^{-1}$ once more to find, for any~$\varepsilon>0$,~$C_{A,\varepsilon}^{(1),'}>0$ independent from~$f$ such that
\begin{equation}
    \label{eq:t1_bound}
|T_1(A,g)| \leq C_{A,\varepsilon}^{(1)'}\|f\|^2_{L^2(\nu)} + \varepsilon\|\nabla^2 g\|^2_{L^2_{\mathrm{HS}}(\nu)}.
\end{equation}

Inserting the estimates~\eqref{eq:t2_bound},~\eqref{eq:t3_bound},~\eqref{eq:t1_bound} into the inequality~\eqref{eq:desired_bound} and setting~$\varepsilon$ to a sufficiently small value finally yields the existence of~$C_{A}>0$ independent from~$f$ such that
\begin{equation}
    \left\|\nabla^2 g\right\|_{L^2_{\mathrm{HS}}(\nu)}\leq C_A\left\|f\right\|,
\end{equation}
which concludes the proof of the boundedness of~$K_A$. In view of the definition~\eqref{eq:def_CD}, both~$C_\lambda$ and~$D_\lambda$ are therefore bounded by~$\mathcal O(\lambda)$ in their respective operator norms as~$\lambda\to +\infty$.

The remainder of the proof is identical to Lemma~\ref{lemma:hypocoercivity}, in view of~\eqref{eq:micro_coercivity_constant} and the bounds on~$C_\lambda$ and~$D_\lambda$.
We obtain the equivalents of the estimates~\eqref{eq:finalbound_A},~\eqref{eq:finalbound_B} and~\eqref{eq:final_bound_grad_p} in exactly the same way.

\begin{lemma}[Well-posedness of~\eqref{eq:underdamped_friction_dependent} and~\eqref{eq:overdamped_friction_dependent}]
    \label{lemma:wp}
    Suppose that Assumptions~\ref{hyp:smoothness} is satisfied. Then for any~$\lambda>0$, there exists globally-defined continuous solutions to the SDEs~\eqref{eq:underdamped_friction_dependent}. If Assumption~\ref{hyp:UVreg} holds, the same conclusion holds for the SDE~\eqref{eq:overdamped_friction_dependent}. These solutions are unique, up to indistinguishability.
\end{lemma}
\begin{proof}
    Let~$\lambda>0$. By Assumption~\ref{hyp:smoothness}, the coefficients of the SDE~\eqref{eq:underdamped_friction_dependent} are locally Lipschitz, and strong solutions are therefore defined up to an explosion time. To extend this result to global-in-time solutions, we use a sufficient Lyapunov criterion, see~\cite[Theorem 3.5]{K12} or~\cite[Theorem 5.9]{RB06}.
    Recalling the definition of the Hamiltonian~$H(q,p)=V(q)+U(p)$, we compute
    \begin{equation}
       \cL_\lambda H = \lambda \cLou H = - \lambda \nabla U^\top D^{-1}\nabla U + \frac{\lambda}{\beta}D^{-1}:\nabla^2 U \leq \frac{\lambda}{\beta}\|D^{-1}\|_{L^\infty_{\mathrm{HS}}}\|\nabla^2 U\|_{L^\infty_{\mathrm{HS}}}.
    \end{equation}
    Since~$H$ has precompact sublevel sets (by Assumption~\ref{hyp:smoothness}), there exists~$c_0>0$ such that~$H(q,p)\geq - c_0 $ for all~$(q,p)\in \cX\times\R^d$, whereby the Lyapunov function
    \begin{equation}
        \cW = 2+\frac{H}{c_0} : \cX\times\R^d \to [1,+\infty)
    \end{equation}
    satisfies
    \begin{equation}
        \cL_\lambda \cW \leq \frac{\lambda}{\beta c_0}\|D^{-1}\|_{L^\infty_{\mathrm{HS}}}\|\nabla^2 U\|_{L^\infty_{\mathrm{HS}}}\cW,\qquad |\cW(q,p)|\xrightarrow{(q,p)\to \infty} + \infty.
    \end{equation}
    Applying Khasminskii's criterion~\cite[Theorem 3.5]{K12}, we obtain the claim for the underdamped dynamics~\eqref{eq:underdamped_friction_dependent}.
    
    The case of the limiting dynamics~\eqref{eq:overdamped_friction_dependent}, which also has locally Lipschitz coefficients, follows from Assumption~\ref{hyp:UVreg}, since the Poincar\'e inequality~\eqref{eq:poincare} for~$\nu$ implies the existence of a Lyapunov function (see for instance~\cite[Theorem 2.3]{CGZ13}), and in turn the non-explosivity of the dynamics~\eqref{eq:overdamped_friction_dependent}.
\end{proof}

\begin{lemma}[Sufficient condition for hypoellipticity]
    \label{lemma:hypoellipticity}
    Suppose Assumptions~\ref{hyp:smoothness} and~\ref{hyp:elliptic} are satisfied, and furthermore that~$\nabla^2 U$ is everywhere non-degenerate. Then Assumption~\ref{hyp:hypoellipticity} is satisfied.
\end{lemma}
\begin{proof}
    We write
    \begin{equation}
       \cL_\lambda = \cA_{0,\lambda} + \sum_{j=1}^d \cA_{j,\lambda}^\dagger \cA_{j,\lambda},\quad \forall\,1\leq j\leq d,\quad \cA_{j,\lambda} = \sqrt{\frac{\lambda}{\beta}}\left(D^{-1/2}(q)\nabla_p\right)_j,
    \end{equation}
    where~``$\dagger$'' denotes the~$L^2(\cX\times\R^d)$ formal adjoint, and
    \begin{equation}
    \cA_{0,\lambda}=\cLham -\lambda \nabla U^\top D^{-1}(q)\nabla_p.
    \end{equation}
    We check that the vector fields~$(\cA_{j,\lambda})_{0\leq j\leq d}$ satisfy H\"ormander's bracket condition (see~\cite{H67}), namely that their Lie algebra has full rank~$2d$ everywhere.

    Computing, for~$1\leq \alpha\leq d$, the commutator
    \begin{equation}
        \begin{aligned}
            \left[\cA_{0,\lambda},\cA_{\alpha,\lambda}\right] &= \sqrt{\frac{\lambda}{\beta}}\sum_{i,j,k} \left[(\partial_i U)\partial_{q_i}-\left(\partial_i V\right)\partial_{p_i}-\lambda\left(\partial_i U\right) D^{-1}_{ij}\partial_{p_j},D^{-1/2}_{\alpha k}\partial_{p_k}\right]\\
            &=\sqrt{\frac{\lambda}{\beta}}\sum_{i,j,k} \left[(\partial_i U)\partial_{q_i}-\lambda\left(\partial_i U\right) D^{-1}_{ij}\partial_{p_j},D^{-1/2}_{\alpha k}\partial_{p_k}\right]\\
            &=-\sum_{i,k}\sqrt{\frac{\lambda}{\beta}}(\partial^2_{ki}U)D^{-1/2}_{\alpha k}\partial_{q_i} + \mathfrak{X}_{\alpha,\lambda}\\
            &=-\sqrt{\frac{\lambda}{\beta}}\left(D^{-1/2}\nabla^2 U\nabla_q\right)_\alpha + \mathfrak{X}_{\alpha,\lambda}
        \end{aligned}
    \end{equation}
    for some~$\mathfrak{X}_{\alpha,\lambda}\in\,\mathrm{Span}(\partial_{p_i},1\leq i\leq d)$. Since~$D^{-1/2}$ has full rank everywhere by Assumption~\ref{hyp:elliptic}, it follows that~$\mathrm{Span}(\partial_{q_i},\partial_{p_i},1\leq i\leq d) = \mathrm{Span}(\cA_{i,\lambda},[\cA_{0,\lambda},\cA_{i,\lambda}],1\leq i\leq d)$ whenever~$\nabla^2 U$ has full rank.
\end{proof}
Even if~$\nabla^2 U$ is degenerate at some points, one may still be able to show the hypoellipticity of the generator by considering higher-order brackets, which leads to sufficient conditions involving higher derivatives of~$U$, see the discussion in~\cite[Appendix A]{BS25}.

\begin{lemma}[Index computation]
    \label{lemma:commutation}
    Let~$A:\cX\to\R^{d\times d}$ be a smooth matrix field such that~$A^{-\top}$ has gradient rows, and~$v(q,p)=A(q)p$. Then the relation~\eqref{eq:symmetry_condition} holds.
\end{lemma}
\begin{proof}
    Since~$A^{-1}$ has gradient columns, we may write locally~$(A^{-1})_{ij} = \partial_i \phi_j$ for some~${(\phi_j)_{1\leq j\leq d}\in\mathcal C^\infty(\cX)^d}$.
    We compute, with derivatives taken with respect to the~$q$-variable:
    \begin{equation}
        \begin{aligned}
            (A^{-1}\nabla_q v)_{ij} &= \sum_{\alpha,\beta} (A^{-1})_{i\alpha}\partial_j(A_{\alpha\beta})p_\beta\\
            &=-\sum_{\alpha,\beta}(A^{-1})_{i\alpha}\left[A\partial_j(A^{-1})A\right]_{\alpha\beta}p_\beta\\
            &=-\sum_{\gamma,\beta}\partial_j(A^{-1})_{i\gamma}A_{\gamma\beta}p_\beta\\
            &=-\sum_{\gamma}\partial^2_{ji}\phi_\gamma v_\gamma,
        \end{aligned}
    \end{equation}
    using the matrix identity~$\partial_j A = -A\partial_j(A^{-1})A$. This shows that~$A^{-1}\nabla_q v$ is symmetric, and therefore that~\eqref{eq:symmetry_condition} holds.
\end{proof}

\paragraph{Acknowledgements}
The authors thanks Tony Leli\`evre, Laurent Michel, Mathias Rousset and Gabriel Stoltz for giving helpful feedback on a preliminary version of this work, and bringing the dynamics~\eqref{eq:effective_ul} to our attention.
 This work received funding from the European Research Council (ERC) under the
    European Union's Horizon 2020 research and innovation program (project
    EMC2, grant agreement No 810367).

\bibliographystyle{abbrv}
\bibliography{bibliography.bib}
\end{document}